\documentclass[11pt]{amsart}
\usepackage{amsfonts}

\usepackage{amscd}
\usepackage{graphicx}
\usepackage{amsmath}
\usepackage{amsbsy}
\usepackage{amssymb}
\newtheorem{theorem}{Theorem}
\theoremstyle{plain}

\newtheorem{corollary}{Corollary}

\newtheorem{definition}{Definition}

\newtheorem{lemma}{Lemma}

\newtheorem{proposition}{Proposition}
\newtheorem{remark}{Remark}

\numberwithin{equation}{section}

\input{tcilatex}

\begin{document}
\title[Non-Central Limit Theorems]{Rosenblatt distribution
subordinated to Gaussian random fields
 with long-range dependence}
 \author{N.N. Leonenko}
\address[N.N. Leonenko]{Cardiff School of Mathematics, Senghennydd Road, Cardiff CF24 4AG, United Kingdom}
\email{LeonenkoN@cardiff.ac.uk}
 
\author{M.D. Ruiz-Medina}
\address[M.D. Ruiz-Medina]{University of Granada\\
Department of Statistics and Operations Research\\
Campus Fuente Nueva s/n, E-18071 Granada \\
Spain}
\email{mruiz@ugr.es}

\author{M.S. Taqqu}
\address[M.S. Taqqu]{Department of Mathematics and Statistics, 111 Cummington St., Boston
University, Boston, MA 02215, USA}
\email{murad@bu.edu}
\date{27 May 2016}
\thanks{This work has been supported in part by project MTM2015--71839--P
(co-funded by European Regional Development Funds), 
  MINECO,  Spain. This research was
also supported under Australian Research Council's Discovery
Projects funding scheme (project number DP160101366), and  under
Cardiff Incoming Visiting Fellowship Scheme and International
Collaboration Seedcorn Fund.
 Murad S. Taqqu was
supported in part by the NSF grants  DMS-1007616 and DMS-1309009 at
Boston University.}

\begin{abstract}
The Karhunen-Lo\`eve expansion and the Fredholm determinant formula
are used, to derive an asymptotic Rosenblatt-type distribution of a
sequence of
integrals of quadratic functions of Gaussian stationary random fields on $%
\mathbb{R}^{d}$ displaying long-range dependence. This distribution
reduces to the usual Rosenblatt distribution when $d=1.$ Several
properties of this new distribution are obtained. Specifically, its
series representation, in terms of independent chi-squared random
variables, is established. Its L\'evy-Khintchine representation, and
 membership to the Thorin subclass of self-decomposable
distributions are obtained as well. The existence and boundedness of
its probability density  then follow as a direct consequence.

\end{abstract}

\maketitle
\noindent \textit{Keywords.} Fredholm determinant; Hermite polynomials;  infinite
divisible distributions; multiple Wiener-It{\^o} stochastic
integrals; non-central limit theorems; Rosenblatt-type
distribution.

\medskip

\noindent \textit{AMS subject classifications.} \ 60F99;  60E10;   60G15;  60G60. 
\section{Introduction}

The aim of this paper is to derive and study the properties of the limit
distribution, as $T\longrightarrow \infty ,$ of the random integral

\begin{equation}
S_{T}=\frac{1}{d_{T}}\int_{D(T)}(Y^{2}(\mathbf{x})-1)d\mathbf{x},
\label{eq2}
\end{equation}%
\noindent where the normalizing function $d_{T}$ is given by
\begin{equation}
d_{T}=T^{d-\alpha }\mathcal{L}(T),\quad 0<\alpha <d/2,  \label{nf}
\end{equation}%
\noindent with $\mathcal{L}$ being a positive slowly varying function at
infinity, that is
\begin{equation}
\lim_{T\rightarrow \infty }\mathcal{L}(T\Vert \mathbf{x}%
\Vert )/\mathcal{L}(T)=1,\label{eq2b}
\end{equation} \noindent for every $\Vert \mathbf{x}\Vert >0,$ and $D(T)\subset \mathbb{R}^{d}$  denotes a homothetic transformation of a set $D\subset
\mathbb{R}^{d},$ with center at the point $\mathbf{0}\in D,$ and
coefficient or scale factor $T>0.$ In the subsequent development,
$D$ is assumed to be a regular bounded domain, whose interior has
positive Lebesgue measure, and
with boundary having null Lebesgue measure. Here, $\{Y(\mathbf{x}),\ \mathbf{%
x}\in \mathbb{R}^{d}\}$ is a zero-mean Gaussian homogeneous and isotropic
random field with values in $\mathbb{R},$ displaying long-range dependence.
That is, $Y$ is assumed to satisfy the following condition:

\medskip

\noindent \textbf{Condition A1}. The random field $\{Y(\mathbf{x}),\ \mathbf{%
x}\in \mathbb{R}^{d}\}$ is a measurable zero-mean Gaussian homogeneous and
isotropic mean-square continuous random field on a probability space $%
(\Omega ,\mathcal{A},P),$ with $\mathrm{E}Y^{2}(\mathbf{x})=1,$ for all $%
\mathbf{x}\in \mathbb{R}^{d},$ and correlation function $\mathrm{E}[Y(%
\mathbf{x})Y(\mathbf{y})]$  $=B(\Vert \mathbf{x}-\mathbf{y}\Vert )$ of the form:
\begin{equation}
B(\Vert \mathbf{z}\Vert )=\frac{\mathcal{L}(\Vert \mathbf{z}\Vert
)}{\Vert \mathbf{z}\Vert ^{\alpha }},\quad \mathbf{z}\in
\mathbb{R}^{d},\quad 0<\alpha < d/2.  \label{cov}
\end{equation}%
\noindent From \textbf{Condition A1}, the correlation  $B$ of $Y$ is
a continuous function of $r=\Vert \mathbf{z}\Vert.$ It then follows that $\mathcal{L}(r)=\mathcal{O}(r^{\alpha }),$ $%
r\longrightarrow 0.$ Note that the covariance function

\begin{equation}
B(\Vert \mathbf{z}\Vert )=\frac{1}{(1+\Vert \mathbf{z}\Vert ^{\beta
})^{\gamma }},\quad 0<\beta \leq 2,\quad \gamma >0,\label{eqcorrfunctex}
\end{equation}%
\noindent is a particular case of the family of covariance functions (\ref%
{cov}) studied here with $\alpha =  \beta\gamma,$ and
\begin{equation}\mathcal{L}(\Vert
\mathbf{z}\Vert )=\Vert \mathbf{z}\Vert ^{\beta \gamma }/(1+\Vert \mathbf{z}%
\Vert ^{\beta })^{\gamma }.\label{svfex}
\end{equation}

The limit random variable of (\ref{eq2}) will be denoted as $S_{\infty }.$
The distribution of $S_{\infty }$ will be referred to as the \textit{%
Rosenblatt-type} distribution, or sometimes simply as the
\textit{Rosenblatt} distribution because this is how it is known in
the case $d=1.$ In that case, a discretized version in time of the
integral (\ref{eq2}) first appears in \cite{Rosenblatt},
 and the limit functional version is considered in \cite{Taqqu75} in the form of the Rosenblatt process. In this classical
setting, the limit of (\ref{eq2}) is represented by a double
Wiener-It{\^o} stochastic integral (see \cite{Dobrushin};
\cite{Taqqu79}). Other relevant references include, for example, \cite{Albin},\cite{Anh}, \cite{Fox},
\cite{Ivanov}, \cite{LeonenkoTaufer06}, \cite{Rosenblatt79}, to mention just a few. The general approach considered here
for deriving the weak-convergence to the Rosenblatt distribution is
inspired by  \cite{Taqqu75}, which is based on the
convergence of characteristic functions. This approach has also been
used, recently, in   \cite{LeonenkoTaufer06}, to study
the characteristic functions of quadratic forms of
strongly-correlated Gaussian random variables sequences.

We suppose here $d\geq 2,$ and thus consider integrals of quadratic
functions of long-range dependence  zero-mean Gaussian stationary
random fields. We pursue, however, a different methodology than in
the case $d=1,$ which was based on the discretization of the
parameter space. A direct extension of these techniques is not
available when $d\geq 2$. Instead of discretizing the parameter
space of the random field, we focus on the characteristic function
for quadratic forms for Hilbert-valued Gaussian random variables
(see, for example, \cite{Da Prato}), and take advantage
of functional analytical tools, like the Karhunen-Lo\`eve expansion
and the Fredholm determinant formula, to obtain the convergence in
distribution to a limit random variable $S_{\infty }$ with
Rosenblatt-type  distribution.

The double Wiener-It{\^o} stochastic integral representation of
$S_{\infty }$  in the spectral domain leads to its series expansion
in terms of independent chi-squared random variables, weighted by
the eigenvalues of the integral operator introduced in equation
(\ref{RKD}) below. The asymptotics of these eigenvalues is given in
Corollary \ref{lfes}.
% from the application of Blumenthal and Getoor (1959) results in Theorem \ref{WaeK}.
The infinitely divisible property  of $S_{\infty }$ is then obtained
as a direct consequence of the previous results derived, in relation
to the series expansion of  $S_{\infty },$ and the asymptotic
properties of the eigenvalues. We also prove that the distribution
of $S_{\infty }$ is self-decomposable, and that it belongs, in
particular, to the Thorin
subclass. The existence and boundedness of the probability density of $%
S_{\infty }$ then  follows.

The outline of the paper is now described. In Section \ref{sec21},
we recall the Karhunen-Lo\`{e}ve expansion, introduce the Fredholm
determinant formula,
and use the referred tools to obtain the characteristic function of (\ref%
{eq2}). In Section \ref{23}, we prove the weak convergence of
(\ref{eq2}) to the random variable $S_{\infty }$ with a
Rosenblatt-type distribution. The double Wiener-It{\^{o}} stochastic
integral representation of $S_{\infty },$ its series expansion in
terms of independent chi-square random variables, and the
asymptotics of the involved eigenvalues are established in Section
\ref{proprd}. These results are applied in Section \ref{sec7} to
derive some properties of the Rosenblatt distribution, e.g.,
infinitely divisible property, self-decomposability, and, in
particular, the membership to the Thorin subclass. Appendices A-C
provide some auxiliary results and the proofs of some propositions
and corollaries.

In this paper we consider the case of real-valued random fields. In what
follows we use the symbols $C,C_{0},M_{1},M_{2},$ etc., to denote constants.
The same symbol may be
used for different constants appearing in the text.

\section{Karhunen-Lo\'eve expansion and related results}

\label{sec21} This section introduces some preliminary definitions,
assumptions and  lemmas
hereafter used  in the derivation of the main results of this paper. We start with the Karhunen-Lo{\`{e}}%
ve Theorem for a zero-mean second-order random field
$\{Y(\mathbf{x}),\ \mathbf{x}\in K\subset \mathbb{R}^{d}\},$ with
continuous covariance function $B_{0}(\mathbf{x},\mathbf{y})=\mathrm{E}[Y(%
\mathbf{x})Y(\mathbf{y})],\ (\mathbf{x,y)}\in K\times K\subset \mathbb{R}%
^{d}\times \mathbb{R}^{d},$ defined on a compact set $K$ of
$\mathbb{R}^{d}$ (see Section 3.2 in \cite{Adler}). This
theorem provides the following orthogonal expansion of the random
field $Y:$
\begin{eqnarray}
&&Y(\mathbf{x})=\sum_{j=1}^{\infty }\sqrt{\lambda _{j}}\phi _{j}(\mathbf{x}%
)\eta _{j},\quad \mathbf{x}\in K,  \notag  \label{eq3} \\
&&\lambda _{k}\phi _{k}(\mathbf{x})=\int_{K}B_{0}(\mathbf{x},\mathbf{y})\phi
_{k}(\mathbf{y})d\mathbf{y},\quad k\in \mathbb{N}_{\ast },\quad \left\langle
\phi _{i},\phi _{j}\right\rangle _{L^{2}(K)}=\delta _{i,j},\quad i,j\in
\mathbb{N}_{\ast },  \notag \\
&&
\end{eqnarray}%
\noindent where $\eta _{k}=\frac{1}{\sqrt{\lambda
_{k}}}\int_{K}Y(\mathbf{x})\phi _{k}(\mathbf{x})d\mathbf{x},$ for
each $k\geq 1,$ and the convergence holds in the $L^{2}(\Omega
,\mathcal{A},P)$ sense. The eigenvalues of $B_{0}$ are considered to
be arranged in decreasing order of magnitude, that is, $\lambda
_{1}\geq \lambda _{2}\geq \dots \geq \lambda _{k-1}\geq \lambda
_{k}\geq \dots .$ The orthonormality of the eigenfunctions $\phi
_{j},$ $j\in \mathbb{N}_{\ast },$ leads to the uncorrelation of the
random variables $\eta _{j},$ $j\in \mathbb{N}_{\ast },$ with
variance one, since
$$E[\eta _{j}\eta _{k}]=\int_{K}\int_{K}B_{0}(\mathbf{x},\mathbf{y})
\phi_{j}(\mathbf{y})\phi
_{k}(\mathbf{x})d\mathbf{y}d\mathbf{x}=\lambda_{j}\int_{K}\phi_{j}(\mathbf{x})\phi
_{k}(\mathbf{x})d\mathbf{x}=\lambda_{j}\delta_{j,k},$$ \noindent
with $\delta $ denoting the Kronecker delta  function. In the
Gaussian case, they are independent.

For each $T>0,$ let us fix some notation related to the
Karhunen-Lo{\`e}ve expansion of the restriction to the set $D(T)$ of
Gaussian random field $Y,$ with covariance function (\ref{cov}).  By
$R_{Y,D(T)}$ we denote the
covariance operator of $Y$ with covariance kernel $B_{0,T}(\mathbf{x},%
\mathbf{y})=\mathrm{E}[Y(\mathbf{x})Y(\mathbf{y})],$ $\mathbf{x},\mathbf{y}%
\in D(T),$ \noindent which, as an operator from $L^{2}(D(T))$ onto $%
L^{2}(D(T)),$ satisfies
\begin{equation*}
R_{Y,D(T)}(\phi _{l,T})(\mathbf{x})=\int_{D(T)}B_{0,T}(\mathbf{x},\mathbf{y}%
)\phi _{l,T}(\mathbf{y})d\mathbf{y}=\lambda _{l,T}(R_{Y,D(T)})\phi _{l,T}(%
\mathbf{x}),\quad l\in \mathbb{N}_{\ast },
\end{equation*}%
\noindent where, in the following, by $\lambda _{k}(A)$ we will
denote the $k$th eigenvalue of the operator $A.$ In particular,
$\{\lambda _{k,T}(R_{Y,D(T)})\}_{k=1}^{\infty }$ and $\{\phi
_{k,T}\}_{k=1}^{\infty }$ respectively
 denote  the eigenvalues and eigenfunctions of $R_{Y,D(T)},$ for
each $T>0.$ Note that, as commented, $B_{0,T}$ refers to the covariance function of $\{Y(%
\mathbf{x}),\ \mathbf{x}\in D(T)\}$ as a function of $(\mathbf{x},\mathbf{y}%
)\in D(T)\times D(T),$ which, under {\bfseries Condition A1}, defines a
non-negative, symmetric and continuous kernel on  $D(T),$
satisfying the conditions assumed in Mercer's Theorem. Hence, the Karhunen-Lo%
{\`e}ve expansion of random field $Y$ holds on $D(T)$, and its covariance
kernel $B_{0,T}$ also admits the series representation
\begin{equation}
B_{0,T}(\mathbf{x},\mathbf{y})=\sum_{j=1}^{\infty }\lambda
_{j,T}(R_{Y,D(T)})\phi _{j,T}(\mathbf{x})\phi _{j,T}(\mathbf{y}),\quad
\mathbf{x},\mathbf{y}\in D(T),  \label{eq4}
\end{equation}%
\noindent where the convergence is absolute and uniform (see, for
example, \cite{Adler}, pp.70-74). The orthonormality of
the eigenfunctions  $\{\phi _{l,T}\}_{l=1}^{\infty }$ yields

\begin{equation}
\frac{1}{d_{T}}\int_{D(T)}Y^{2}(\mathbf{x})d\mathbf{x}=\frac{1}{d_{T}}%
\sum_{j=1}^{\infty }\lambda _{j,T}(R_{Y,D(T)})\eta _{j,T}^{2}.
\label{Ysquare}
\end{equation}

In the derivation of the limit characteristic function of
(\ref{eq2}), we will use the Fredholm determinant formula of a trace
operator. Recall first that a positive operator $A$ on a separable
Hilbert space $H$ is a trace operator if
\begin{equation}
\|A\|_{1}\equiv \mbox{Tr}(A) \equiv \sum_{k}\left\langle
(A^{*}A)^{1/2}\varphi_{k},\varphi_{k}\right\rangle_{H}<\infty ,
\label{eqtr1}
\end{equation}
\noindent where $A^{*}$ denotes the adjoint of $A$ and $\{
\varphi_{k}\}$ is an orthonormal basis of the Hilbert space $H$ (see
\cite{Reed}, pp. 207-209). A sufficient condition for a
compact and self-adjoint operator $A$ to belong to the trace class
is $\sum_{k=1}^{\infty }\lambda_{k}(A)<\infty .$ For each finite
$T>0,$ the operator $R_{Y,D(T)}$ is in the trace class, since from
equation (\ref{eq4}), applying the orthonormality of the
eigenfunction system $\{\phi_{j,T},\ j\in \mathbb{N}_{*}\},$ and
keeping in mind that $B_{0,T}(\mathbf{0})=1,$ we have
\begin{equation}
Tr(R_{Y,D(T)})=\sum_{j=1}^{\infty }\lambda
_{j,T}(R_{Y,D(T)})=\int_{D(T)}B_{0,T}(\mathbf{x},\mathbf{x})d\mathbf{x}%
=\int_{D(T)}d\mathbf{x}=T^{d}|D|<\infty,  \label{TR4}
\end{equation}%
\noindent where $|D|$ denotes the Lebesgue measure of the compact
set $D.$ Note that the class of compact and self-adjoint operators
contains the class of trace and self-adjoint operators. Hence, under
{\bfseries Condition A1}, from equation (\ref{TR4}), the restriction
of $Y$ to $D(T)$ admits a Karhunen-Lo\'eve expansion, convergent in
the mean-square sense (i.e., in the
$L^{2}(\Omega,\mathcal{A},P)$-sense), for any $T>0,$ and for an
arbitrary regular  bounded domain $D.$ Furthermore, for any $k\geq
1,$
\begin{equation}
R^{k}_{Y,D(T)}f(\mathbf{x})=\int_{D(T)}B_{0,T}^{\ast (k)}(\mathbf{x},\mathbf{%
y})f(\mathbf{y})d\mathbf{y},\quad f\in L^{2}(D(T)),  \label{ire}
\end{equation}
\noindent where $B_{0,T}^{\ast (k)}$ denotes
\begin{equation*}
B_{0,T}^{\ast (1)}(\mathbf{x},\mathbf{y})=B_{0,T}(\mathbf{x},\mathbf{y}%
),\quad k=1,
\end{equation*}
\begin{equation}
B_{0,T}^{\ast (k)}(\mathbf{x},\mathbf{y})=\int_{D(T)}B_{0,T}^{\ast (k-1)}(%
\mathbf{x},\mathbf{z})B_{0,T}(\mathbf{z},\mathbf{y})d\mathbf{z},\quad
k=2,3,\dots.  \label{TR2}
\end{equation}
\noindent From equations (\ref{eq4}) and (\ref{TR2}), applying the orthonormality of $%
\phi_{j,T},$ $j\in \mathbb{N}_{*},$ one can obtain
\begin{equation}
Tr (R^{k}_{Y,D(T)})=\sum_{j=1}^{\infty }\lambda
_{j,T}^{k}(R_{Y,D(T)})=\int_{D(T)}B_{0,T}^{\ast (k)}(\mathbf{x},\mathbf{x})d%
\mathbf{x}<\infty , \quad k\in \mathbb{N}_{*},  \label{TR1}
\end{equation}

\noindent since, for every $k\geq 1,$ $|\lambda_{k}(R_{Y,D(T)})|\leq
M|\lambda_{k}(R_{Y,D(T)})|^{k}=M|\lambda_{k}(R^{k}_{Y,D(T)})|,$ for
some positive constant $M.$ In particular, in the homogeneous random
field case,
\begin{eqnarray}  \label{TR3}
Tr(R^{k}_{Y,D(T)})&=&\sum_{j=1}^{\infty }\lambda
_{j,T}^{k}(R_{Y,D(T)})=\int_{D(T)}B_{0,T}^{\ast (k)}(\mathbf{x}_{k},\mathbf{x}_{k})d%
\mathbf{x}_{k}  \notag \\
&=&\int_{D(T)}...\int_{D(T)}\left[ \prod_{j=1}^{k-1}B_{0,T}(\mathbf{x}_{j+1}-%
\mathbf{x}_{j})\right] B_{0,T}(\mathbf{x}_{1}-\mathbf{x}_{k})d\mathbf{x}%
_{1}\dots d\mathbf{x}_{k},  \notag \\
\end{eqnarray}
\noindent and, in the homogeneous and isotropic case, for $k=2,$
\begin{equation}
Tr(R^{2}_{Y,D(T)})=\sum_{j=1}^{\infty }\lambda _{j,T}^{2}(R_{Y,D(T)})=
\int_{D(T)}\int_{D(T)}\frac{\mathcal{L}^{2}(\|\mathbf{x}-\mathbf{y}\|)}{\|%
\mathbf{x}-\mathbf{y}\|^{2\alpha }}d\mathbf{y}d\mathbf{x}.  \label{TR5b}
\end{equation}

The following definition introduces the Fredholm determinant of an
operator $A,$ as a complex-valued function which generalizes the
determinant of a matrix.

\begin{definition}
\label{def1} (see, for example, \cite{Simon},  Chapter 5, pp.47-48,
equation (5.12)) Let $A$ be a trace operator on a separable Hilbert
space $H.$ The Fredholm determinant of $A$ is
\begin{equation}
\mathcal{D}(\omega )=\mbox{det}(I-\omega A)=\exp\left(-\sum_{k=1}^{\infty }%
\frac{\mbox{Tr}A^{k}}{k}\omega^{k}\right)=\exp\left(-\sum_{k=1}^{\infty
}\sum_{l=1}^{\infty}[\lambda_{l}(A)]^{k}\frac{\omega^{k}}{k}\right),
\label{fdf}
\end{equation}
\noindent for $\omega \in \mathbb{C},$ and $|\omega |\|A\|_{1}< 1.$ Note
that $\| A^{m}\|_{1}\leq \|A\|_{1}^{m},$ for $A$ being a trace operator.
\end{definition}

\begin{lemma}
\label{lemm0} \textit{Let $\{Y(\mathbf{x}),\ \mathbf{x}\in D\subset
\mathbb{R}^{d}\}$ be an  integrable and continuous, in the
mean-square sense, zero-mean, Gaussian random field, on a bounded
regular domain  $D\subseteq \mathbb{R}^{d}$ containing the point
zero. Then, the following identity holds:}
\begin{eqnarray}
\mathrm{E}\left[ \exp \left( \mathrm{i}\xi \int_{D}Y^{2}(\mathbf{x})d\mathbf{%
x}\right) \right] &=&\prod_{j=1}^{\infty }(1-2\lambda _{j}(R_{Y,D})\mathrm{i}%
\xi )^{-1/2}=(\mathcal{D}(2\mathrm{i}\xi ))^{-1/2}  \notag \\
&=&\exp \left( \frac{1}{2}\sum_{m=1}^{\infty }\frac{(2\mathrm{i}\xi )^{m}}{m}%
\mbox{Tr}(R_{Y,D}^{m})\right) ,  \label{lemm1}
\end{eqnarray}%
\noindent \textit{for $\Vert R_{Y,D}\Vert _{1}|2\mathrm{i}\xi |<1,$
as given in Definition \ref{def1}.}
\end{lemma}

\begin{proof}
 The covariance
operator $R_{Y,D}$ of $Y,$  acting on the space  $L^{2}(D),$ is in
the trace class. From Definition \ref{def1},  the following
identities
  hold:
\begin{eqnarray}
& & E\left[ \exp\left(\mathrm{i}\xi \int_{D}Y^{2}(\mathbf{x})d\mathbf{x}\right)\right]= E\left[ \exp\left(\mathrm{i}\xi
\sum_{j=1}^{\infty}\lambda_{j}(R_{Y,D})\eta_{j}^{2}\right) \right]
\nonumber\\
& &=\prod_{j=1}^{\infty}E\left[\exp\left(\mathrm{i}\xi
\lambda_{j}(R_{Y,D})\eta_{j}^{2}\right)\right]=\prod_{j=1}^{\infty}(1-2\lambda_{j}(R_{Y,D})\mathrm{i}\xi)^{-1/2}=(\mathcal{D}(2\mathrm{i}\xi ))^{-1/2}\nonumber\\
& & =\left[\exp\left( -\sum_{m=1}^{\infty }\frac{(2\mathrm{i}\xi)^{m}}{m}\mbox{Tr}%
(R_{Y,D}^{m})\right)\right]^{-1/2}=\exp\left( \frac{1}{2}\sum_{m=1}^{\infty }\frac{(2\mathrm{i}\xi )^{m}}{m}\mbox{Tr}%
(R_{Y,D}^{m})\right),\nonumber\\ \label{Fredet}
\end{eqnarray}
\noindent   where the last two identities in equation (\ref{Fredet})
are finite for $|\xi |<\frac{1}{2|D|},$  from the Fredholm
determinant formula (\ref{fdf}). Note that
\begin{eqnarray}
& &\mbox{Tr}(R_{Y,\mathcal{D}}^{m})=\sum_{j=1}^{\infty
}\lambda_{j}^{m}(R_{Y,D})\leq
\lambda_{1}^{m-1}(R_{Y,D})\sum_{j=1}^{\infty }\lambda_{j}(R_{Y,D})
=\lambda_{1}^{m-1}(R_{Y,D})\|R_{Y,\mathcal{D}}\|_{1} <\infty .
\end{eqnarray}
  \end{proof}

\begin{remark}
\label{rem2} \textrm{{Similarly to equation (\ref{lemm1}), one can
obtain the following identities, which will be used in the
subsequent development: For a homothetic transformation $D\left(
T\right) $ of  $D \subset \mathbb{R}^{d},$ with center at the point
$\mathbf{0} \in D,$ and coefficient $T>0,$
\begin{eqnarray}
E\left[ \exp\left(\mathrm{i}\xi \int_{D(T)}Y^{2}(\mathbf{x})d\mathbf{x}\right)\right]
&=&\prod_{j=1}^{\infty}(1-2\lambda_{j,T}(R_{Y,D(T)})\mathrm{i}\xi)^{-1/2}=(\mathcal{D}%
_{T}(2\mathrm{i}\xi))^{-1/2}  \notag \\
&=& \exp\left( \frac{1}{2}\sum_{m=1}^{\infty }\frac{(2\mathrm{i}\xi)^{m}}{m}\mbox{Tr}%
(R_{Y,D(T)}^{m})\right),  \label{expands}
\end{eqnarray}
\noindent where  $\lambda _{1,T}(R_{Y,D(T)})\geq
\lambda_{2,T}(R_{Y,D(T)})\geq \dots \geq \lambda
_{j,T}(R_{Y,D(T)})\geq \dots ,$ with, as before,
$\{\lambda_{j,T}(R_{Y,D(T)}),\ j\in \mathbb{N}_{*}\}$ denoting the
system of eigenvalues of the covariance operator $R_{Y,D(T)}$ of
$Y,$ as an operator from $L^{2}(D(T))$ onto $L^{2}(D(T)).$ \noindent
The last identity in equation (\ref{expands})
holds for $\|R_{Y,D(T)}\|_{1}|2\mathrm{i}\xi |<1,$ i.e., for $\mbox{Tr}%
(R_{Y,D(T)})|2\mathrm{i}\xi |= T^{d}|D ||2\mathrm{i}\xi |<1,$ or
equivalently for $|\xi |<\frac{1}{2T^{d}|D|}.$ } }
\end{remark}

\section{Weak convergence of the random integral $S_T$}

\label{23}

This section   provides
 the weak convergence of the random integral (\ref{eq2}) to a
Rosenblatt-type distribution, in Theorem \ref{pr1}.  This results is
based on the asymptotic behavior of the eigenvalues of the integral
operator $\mathcal{K}_{\alpha }$ (see Theorem \ref{WaeK} below)
\begin{equation}
\mathcal{K}_{\alpha }(f)(\mathbf{x})=\int_{D }\frac{1}{\|\mathbf{x}-\mathbf{y%
}\|^{\alpha }}f(\mathbf{y})d\mathbf{y},\quad \forall f\in \mbox{Supp}(%
\mathcal{K}_{\alpha}),  \quad 0<\alpha <d, \label{RKD}
\end{equation}
\noindent with $\mbox{Supp}(A)$ denoting the support of operator
$A.$ Operator (\ref{RKD}) can be related with   the  Riesz potential
$(-\Delta )^{-\beta /2}$ of order $\beta ,$ $0<\beta< d,$ on
$\mathbb{R}^{d},$ formally  defined  as (see \cite{Stein}, p.117)
\begin{equation}
(-\Delta )^{-\beta /2}(f)(\mathbf{x})=\frac{1}{\gamma (\beta
)}\int_{\mathbb{R}^{d}}\|\mathbf{x}-\mathbf{y}\|^{-d+\beta}f(\mathbf{y})d\mathbf{y},
\label{rps117}
\end{equation}
\noindent where  $(-\Delta )$ denotes the negative Laplacian
operator, and \begin{equation}\gamma (\beta
)=\frac{\pi^{d/2}2^{\beta }\Gamma (\beta/2)}{\Gamma
\left(\frac{d-\beta}{2}\right)}=\frac{1}{c(d,\beta )},\quad 0<\beta
<d.\label{eRc}
\end{equation}

\noindent Indeed, except  a constant, the function
$\left(1/\|\mathbf{x}-\mathbf{y}\|^{\alpha }\right)$ in equation
(\ref{RKD}) defines the kernel of the Riesz potential $(-\Delta
)^{(\alpha -d)/2}$ of order $\beta =(d-\alpha ),$ for $0<\alpha <d.$
   Similarly,
$\left(1/\|\mathbf{x}-\mathbf{y}\|^{2\alpha }\right)$ is the kernel
of the Riesz potential $(-\Delta )^{\alpha -d/2}$ of order $\beta
=(d-2\alpha )$ on $\mathbb{R}^{d},$ for $0<\alpha < d/2.$

Recall that the Schwartz space $\mathcal{S}(\mathbb{R}^{d})$ is the
space of of  infinitely differentiable functions on
$\mathbb{R}^{d},$  whose derivatives remain bounded when multiplied
by polynomials, i.e., whose derivatives are rapidly decreasing.
Particularly,  $C^{\infty}_{0}(D)\subset
\mathcal{S}(\mathbb{R}^{d}),$ with $C^{\infty}_{0}(D)$ denoting the
infinitely differentiable functions with compact support contained
in $D.$

The Fourier transform of the Riesz potential is understood in the
weak sense, considering the space  $\mathcal{S}(\mathbb{R}^{d}).$
The  following lemma provides such a transform (see Lemma 1 of
\cite{Stein},  p.117):
\begin{lemma}
\label{l1s117} \textit{Let us consider $0<\beta <d.$
\begin{itemize}
\item[(i)] The Fourier transform of the function
$\|\mathbf{z}\|^{-d+\beta}$ is $\gamma (\beta )\|\mathbf{z}\|^{-\beta },$ in the sense that
\begin{equation}
\int_{\mathbb{R}^{d}}\|\mathbf{z}\|^{-d+\beta }\overline{\psi(\mathbf{z})}d\mathbf{z}= \int_{\mathbb{R}^{d}}\gamma
(\beta )\|\mathbf{z}\|^{-\beta }\overline{\mathcal{F}(\psi )(\mathbf{z})}d\mathbf{z},\quad \forall \psi \in
\mathcal{S}(\mathbb{R}^{d}), \label{fist117}
\end{equation}
\noindent where $$\mathcal{F}(\psi
)(\mathbf{z})=\int_{\mathbb{R}^{d}}\exp\left(-\mathrm{i}\left\langle
\mathbf{x},\mathbf{z}\right\rangle \right)\psi
(\mathbf{x})d\mathbf{x}$$ \noindent denotes the Fourier transform of
$\psi.$
\item[(ii)] The identity $\mathcal{F}\left((-\Delta
)^{-\beta /2}(f)\right)(\mathbf{z})=\|\mathbf{z}\|^{-\beta }\mathcal{F}(f)(\mathbf{z})$ holds in the sense that
\begin{equation}
\int_{\mathbb{R}^{d}} (-\Delta )^{-\beta
/2}(f)(\mathbf{x})\overline{g(\mathbf{x})}d\mathbf{x}=\frac{1}{(2\pi)^{d}}\int_{\mathbb{R}^{d}}\mathcal{F}(f)(\mathbf{x})\|\mathbf{x}\|^{-\beta
}\overline{\mathcal{F}(g)(\mathbf{x})}d\mathbf{x},\quad \forall f,g\in \mathcal{S}(\mathbb{R}^{d}). \label{irpp}
\end{equation}
\end{itemize}}
\end{lemma}

In particular, the following convolution formula is obtained by
iteration of (\ref{irpp}) using (\ref{rps117}):
\begin{eqnarray}
& & \int_{\mathbb{R}^{d}}\left(\frac{1}{\gamma (\beta
)}\int_{\mathbb{R}^{d}}\|\mathbf{x}-\mathbf{y}\|^{-d+\beta}
\left[\frac{1}{\gamma (\beta
)}\int_{\mathbb{R}^{d}}\|\mathbf{y}-\mathbf{z}\|^{-d+\beta}f(\mathbf{z})d\mathbf{z}\right]d\mathbf{y}\right)\overline{g(\mathbf{x})}
d\mathbf{x} \nonumber\\
&&= \int_{\mathbb{R}^{d}}(-\Delta )^{-\beta /2}\left[(-\Delta
)^{-\beta /2}(f)\right](\mathbf{x})\
\overline{g(\mathbf{x})}d\mathbf{x} \nonumber\\ &
=&\frac{1}{(2\pi)^{d}}\int_{\mathbb{R}^{d}}\left[\mathcal{F}((-\Delta
)^{-\beta /2}(f))(\mathbf{x})\right]\|\mathbf{x}\|^{-\beta
}\overline{\mathcal{F}(g)(\mathbf{x})}d\mathbf{x} \nonumber
\end{eqnarray}
\begin{eqnarray}
  &&=\frac{1}{(2\pi)^{d}}
\int_{\mathbb{R}^{d}}\mathcal{F}(f)(\mathbf{x})\|\mathbf{x}\|^{-\beta }\|\mathbf{x}\|^{-\beta }
\overline{\mathcal{F}(g)(\mathbf{x})}d\mathbf{x} \nonumber\\
&&=\frac{1}{(2\pi)^{d}}\int_{\mathbb{R}^{d}}\mathcal{F}(f)(\mathbf{x})\|\mathbf{x}\|^{-2\beta
}\overline{\mathcal{F}(g)(\mathbf{x})}d\mathbf{x}\nonumber\\
&&= \int_{\mathbb{R}^{d}}(-\Delta )^{-\beta
}(f)(\mathbf{x})\overline{g(\mathbf{x})}d\mathbf{x},\quad \forall
f,g\in \mathcal{S}(\mathbb{R}^{d}),\quad
0<\beta<d/2,\nonumber\\
\label{cfsf}
\end{eqnarray}
\noindent  where we have used that if $f\in \mathcal{S}(\mathbb{R}^{d}),$ then $(-\Delta )^{-\beta /2}(f)\in
\mathcal{S}(\mathbb{R}^{d}).$ From equation  (\ref{cfsf}), and Lemma \ref{l1s117}(i),
\begin{eqnarray}
& & \int_{\mathbb{R}^{d}}\frac{1}{\gamma (2\beta)}\|\mathbf{z}\|^{-d+2\beta }\overline{f(\mathbf{z})}d\mathbf{z}=
\int_{\mathbb{R}^{d}}\|\mathbf{z}\|^{-2\beta }\overline{\mathcal{F}(f
)(\mathbf{z})}d\mathbf{z}\nonumber\\
& &=\int_{\mathbb{R}^{d}}\frac{1}{[\gamma (\beta
)]^{2}}\left[\int_{\mathbb{R}^{d}}\|\mathbf{z}-\mathbf{y}\|^{-d+\beta
}\|\mathbf{y}\|^{-d+\beta
}d\mathbf{y}\right]\overline{f(\mathbf{z})}d\mathbf{z},\quad \forall
f\in \mathcal{S}(\mathbb{R}^{d}),\quad
0<\beta<d/2.\nonumber\\\label{cfsf2}
\end{eqnarray}

Let us now consider on the space of infinitely differentiable
functions with compact support contained in $D,$ $C_{0}^{\infty
}(D)\subset \mathcal{S}(\mathbb{R}^{d}),$  the norm

\begin{eqnarray} &&\|f\|^{2}_{(-\Delta
)^{\alpha -d/2}}
%=(-\Delta )^{\alpha -d/2}(f)(f)
=\left\langle (-\Delta )^{\alpha -d/2}(f),f\right\rangle_{L^{2}(\mathbb{R}^{d})}=\left\langle (-\Delta )^{\alpha -d/2}(f),f\right\rangle_{L^{2}(D)}\nonumber\\
& &
= \int_{\mathbb{R}^{d}}(-\Delta )^{\alpha
-d/2}(f)(\mathbf{x})\overline{f(\mathbf{x})}d\mathbf{x}
=\int_{\mathbb{R}^{d}}\frac{1}{\gamma (d-2\alpha
)}\int_{\mathbb{R}^{d}} \frac{1}{\|\mathbf{x}-\mathbf{y}\|^{2\alpha
}}f(\mathbf{y})\overline{f(\mathbf{x})}d\mathbf{y}d\mathbf{x}\nonumber\\
 && =\frac{1}{(2\pi)^{d}}\int_{\mathbb{R}^{d}} |\mathcal{F}(f)(\boldsymbol{\lambda
})|^{2}\|\boldsymbol{\lambda }\|^{-(d-2\alpha ) }d\boldsymbol{\lambda },\quad \forall f\in
C_{0}^{\infty }(D),\quad 0<\alpha <d/2. \label{eqrkdef}
\end{eqnarray}
 \noindent The associated inner product is given by
\begin{eqnarray}\left\langle f,g\right\rangle_{(-\Delta )^{\alpha
-d/2}} &=&
\int_{\mathbb{R}^{d}}\frac{1}{\gamma (d-2\alpha
)}\int_{\mathbb{R}^{d}} \frac{1}{\|\mathbf{x}-\mathbf{y}\|^{2\alpha
}}f(\mathbf{y})\overline{g(\mathbf{x})}d\mathbf{y}d\mathbf{x}\nonumber\\
&=&
\int_{D}\frac{1}{\gamma (d-2\alpha
)}\int_{D} \frac{1}{\|\mathbf{x}-\mathbf{y}\|^{2\alpha
}}f(\mathbf{y})\overline{g(\mathbf{x})}d\mathbf{y}d\mathbf{x},
\label{ip}
\end{eqnarray}
\noindent for all $f,g\in C_{0}^{\infty }(D).$ The closure of
$C_{0}^{\infty }(D)$ with the norm $\|\cdot\|_{(-\Delta )^{\alpha
-d/2}},$ introduced  in (\ref{eqrkdef}), defines a Hilbert space,
which  will be denoted as
 $\mathcal{H}_{2\alpha -d}=\overline{C_{0}^{\infty
}(D)}^{\|\cdot \|_{(-\Delta )^{\alpha -d/2}}}.$

\begin{remark}
\label{insp} \textrm{For a bounded open domain $D,$ from Proposition
2.2. in  \cite{Caetano}, with $D=n-1,$ $p=q=2,$  and $s=0$ (hence,
$A_{pq}^{s}(D)=A_{22}^{0}(D)=L^{2}(D),$ where, as usual,  $L^{2}(D)$
denotes the space of square integrable functions on $D$), we have}
\begin{equation}\overline{C_{0}^{\infty }(D)}^{\|\cdot
\|_{L^{2}(\mathbb{R}^{d})}}=L^{2}(D),\label{idl2c}
\end{equation}
\noindent \textrm{(see also \cite{Triebel}, for the case of regular
bounded open domains with $C^{\infty}-$boundaries). In addition, for
all $f\in C_{0}^{\infty }(D),$ by definition of the norm
(\ref{eqrkdef}),}
$$\|f\|_{(-\Delta
)^{\alpha -d/2}}\leq C\|f\|_{L^{2}(\mathbb{R}^{d})},$$ \noindent
\textrm{that is, all convergent sequences of $C_{0}^{\infty }(D)$ in
the $L^{2}(\mathbb{R}^{d})$ norm are also convergent in the
$\mathcal{H}_{2\alpha -d}$ norm. Hence, the closure of
$C_{0}^{\infty }(D),$ with respect to the norm $\|\cdot
\|_{L^{2}(\mathbb{R}^{d})},$ is included in the closure of
$C_{0}^{\infty }(D),$ with respect to the norm $\|\cdot\|_{(-\Delta
)^{\alpha -d/2}}.$ Therefore, from equation (\ref{idl2c}),}
\begin{equation}L^{2}(D)=\overline{C_{0}^{\infty
}(D)}^{\|\cdot \|_{L^{2}(\mathbb{R}^{d})}} \subseteq
\overline{C_{0}^{\infty }(D)}^{\|\cdot \|_{(-\Delta )^{\alpha
-d/2}}}=\mathcal{H}_{2\alpha -d}.\label{eics0}
\end{equation}

\end{remark}

 The   asymptotic order of the eigenvalues of operator
$\mathcal{K}_{\alpha}$ on $L^{2}(D),$ in the case $d\geq 2,$ are given in the next
result (see, for
example, \cite{Yang}, \cite{Widom} and \cite{Zahle},
p.197).  (See also \cite{Dostanic} and \cite{Veillette},
for the case $d=1$).

\begin{theorem}
\label{WaeK}

 Let us consider the integral operator
$\mathcal{K}_{\alpha }$ introduced in equation (\ref{RKD}) as an
operator on the space $L^{2}(D),$ with $D$ being a bounded open
domain of $\mathbb{R}^{d}.$ The
 following asymptotic is
satisfied by the eigenvalues  $\lambda_{k}(\mathcal{K}_{\alpha }),$
$k\geq 1,$ of operator $\mathcal{K}_{\alpha }:$
\begin{equation}
\lim_{k\longrightarrow \infty }\frac{\lambda_{k}(\mathcal{K}_{\alpha })}{%
k^{-(d-\alpha )/d}}=\widetilde{c}(d,\alpha )|D |^{(d-\alpha )/d},
\label{eqfa}
\end{equation}
\noindent where $|D|$ denotes, as before, the Lebesgue measure of
domain $D,$ and
\begin{equation}
\widetilde{c}(d,\alpha )= \pi^{\alpha
/2}\left(\frac{2}{d}\right)^{(d-\alpha )/d}\frac{\Gamma
\left(\frac{d-\alpha}{2}\right)}{\Gamma \left(\frac{\alpha
}{2}\right)\left[\Gamma
\left(\frac{d}{2}\right)\right]^{(d-\alpha)/d}}.
  \label{eqctilde}
\end{equation}

\end{theorem}

\begin{proof}

We apply  the  results derived in \cite{Widom}, on the asymptotic
behavior of the eigenvalues associated with certain class of
integral equations. Specifically,  the following integral equation
is considered in that paper:
\begin{equation}\int V^{1/2}(\mathbf{x})k(\mathbf{x}-\mathbf{y})
V^{1/2}(\mathbf{y})f(\mathbf{y})d\mathbf{y}=\lambda f(\mathbf{x}),
\label{eivalue}
\end{equation}

\noindent where $k$ is an integrable function over a Euclidean space
$E_{d}$ of dimension $d,$ having positive Fourier transform, and
where $V$ is a bounded non-negative function with bounded support.
In particular, \cite{Widom} considers the case where
$E_{d}=\mathbb{R}^{d},$ $V$ is the indicator function of a bounded
domain $D\subseteq \mathbb{R}^{d},$ and
$k(\|\mathbf{x}-\mathbf{y}\|)=\|\mathbf{x}-\mathbf{y}\|^{\alpha },$
for $\alpha >-d,$ and $\alpha \neq 0,2,4,\dots.$ Function  $k$
coincides in $\mathbb{R}^{d}\setminus D$ with a function whose
Fourier transform $f(\boldsymbol{\xi})$ is asymptotically equal to
$$2^{d-\alpha }\pi^{d/2}\frac{\Gamma \left(\frac{d-\alpha }{2}\right)}{\Gamma \left(\frac{\alpha
}{2}\right)}|\boldsymbol{\xi}|^{-d+\alpha }$$

\noindent (see also the right-hand side of equation (\ref{fist117})
for $\beta =d-\alpha ,$ with $0<\alpha <d$). For $\alpha >-d,\
\alpha \neq 0,2,4,\dots,$
the following asymptotic of the eigenvalues of the
integral operator  with kernel $k(\|\mathbf{x}-\mathbf{y}\|)=\|\mathbf{x}-\mathbf{y}\|^{\alpha }$ is given in equation (2) in \cite{Widom}:  \begin{equation}\lambda_{k}\sim
\pi^{-\alpha /2}\left(\frac{2}{d}\right)^{\frac{d+\alpha
}{d}}\frac{\Gamma \left(\frac{d+\alpha }{2}\right)}{\Gamma
\left(\frac{-\alpha }{2}\right)\left[\Gamma
\left(\frac{d}{2}\right)\right]^{(d+\alpha
)/d}}\left[\int_{\mathbb{R}^{d}}\left[V(\mathbf{x})\right]^{d/(d+\alpha
)}d\mathbf{x}\right]^{(d+\alpha )/d}k^{-(d+\alpha )/d},\label{eqwae}
\end{equation}
\noindent with
$$\int_{\mathbb{R}^{d}}\left[V(\mathbf{x})\right]^{d/(d-\alpha
)}d\mathbf{x}=|D|.$$ \noindent Since function $k$
 in \cite{Widom} coincides with the kernel of the integral
operator $\mathcal{K}_{\alpha }$  in equation (\ref{RKD}),
for $\alpha \in (-d,0),$ equation (\ref{eqwae}) then
leads to the following asymptotic of the eigenvalues of $\mathcal{K}_{\alpha }:$
$$\lambda_{k}(\mathcal{K}_{\alpha })\sim \pi^{\alpha /2}\left(\frac{2}{d}\right)^{\frac{d-\alpha
}{d}}\frac{\Gamma \left(\frac{d-\alpha }{2}\right)}{\Gamma
\left(\frac{\alpha }{2}\right)\left[\Gamma
\left(\frac{d}{2}\right)\right]^{(d-\alpha
)/d}}\left[\int_{\mathbb{R}^{d}}\left[V(\mathbf{x})\right]^{d/(d-\alpha
)}d\mathbf{x}\right]^{(d-\alpha )/d}k^{-(d-\alpha )/d}.$$

\end{proof}

\begin{remark}
\label{remfv2} \textrm{Similar results to those ones presented in
Theorem 3.2 of \cite{Veillette} can be derived for the
spectral zeta function of
the Dirichlet Laplacian on a bounded closed multidimensional interval of $%
\mathbb{R}^{d}$ (see also \cite{Dostanic}, for the case of $d=1$).
For a continuous function of the negative Dirichlet Laplacian, the
explicit computation of its trace cannot always be obtained in a general regular bounded open domain of $%
\mathbb{R}^{d}.$  Specifically, the knowledge of the eigenvalues is
guaranteed for highly symmetric regions like the  the sphere, or
regions bounded by parallel planes (see, for example, \cite{Muller};
\cite{Park02a}; \cite{Park02b}). In particular, for the torus
$\mathbb{T}^{2}$ in $\mathbb{R}^{2},$ the Spectral Zeta Function can
be explicitly computed (see, for example, \cite{Arendt},
Chapter 1, equation (1.49), pp. 28-29).}
\end{remark}

The following condition is assumed to be satisfied by  the slowly  varying function $\mathcal{L}$  in  Theorem \ref{pr1} below.

\bigskip

\noindent {\bfseries Condition A2.} For every $m\geq 2,$ there exists a
constant $C>0$ such that
\begin{equation*}
\int_{D}..(m).\int_{D}\frac{\mathcal{L}(T\Vert \mathbf{x}_{1}-\mathbf{x}%
_{2}\Vert )}{\mathcal{L}(T)\Vert \mathbf{x}_{1}-\mathbf{x}_{2}\Vert ^{\alpha
}}\frac{\mathcal{L}(T\Vert \mathbf{x}_{2}-\mathbf{x}_{3}\Vert )}{\mathcal{L}%
(T)\Vert \mathbf{x}_{2}-\mathbf{x}_{3}\Vert ^{\alpha }}\cdot \cdot \cdot
\frac{\mathcal{L}(T\Vert \mathbf{x}_{m}-\mathbf{x}_{1}\Vert )}{\mathcal{L}%
(T)\Vert \mathbf{x}_{m}-\mathbf{x}_{1}\Vert ^{\alpha }}d\mathbf{x}_{1}d%
\mathbf{x}_{2}\cdot \cdot \cdot d\mathbf{x}_{m}\leq
\end{equation*}%
\begin{equation}
\leq
C\int_{D}...(m).\int_{D}\frac{d\mathbf{x}_{1}d\mathbf{x}_{2}\cdot
\cdot \cdot d\mathbf{x}_{m}}{\Vert
\mathbf{x}_{1}-\mathbf{x}_{2}\Vert ^{\alpha }\Vert
\mathbf{x}_{2}-\mathbf{x}_{3}\Vert ^{\alpha }\cdot \cdot \cdot \Vert
\mathbf{x}_{m}-\mathbf{x}_{1}\Vert ^{\alpha }}. \label{A2}
\end{equation}

\bigskip Note that {\bfseries Condition A2} is satisfied by slowly varying
functions such that
\begin{equation}
\sup_{T,\mathbf{x}_{1},\mathbf{x}_{2}\in D}\frac{\mathcal{L}(T\Vert \mathbf{x%
}_{1}-\mathbf{x}_{2}\Vert )}{\mathcal{L}(T)}\leq C_{0},  \label{scA2}
\end{equation}
\noindent for $0<C_{0}\leq 1.$ This condition holds for bounded
slowly varying functions as in (\ref{svfex}),  in the case where
$D\subseteq \mathcal{B}_{1}(\mathbf{0}),$
with $\mathcal{B}_{1}(\mathbf{0})=\{ \mathbf{x}\in \mathbb{R}^{d},\ \|\mathbf{x}%
\|\leq 1\}.$

For the derivation of the limit distribution, when $T\longrightarrow \infty ,$
of the functional (\ref{eq2}), we first compute its variance, in terms of $%
H_{2},$ the Hermite polynomial of order $2.$ It is well-known that Hermite
polynomials form a complete orthogonal system of the Hilbert space $L_{2}(%
\mathbb{R},\varphi (u)du),$ the space of square integrable functions with
respect to the standard normal density $\varphi .$ They are defined as
follows:
\begin{equation*}
H_{k}(u)=(-1)^{k}e^{\frac{u^{2}}{2}}\frac{d^{k}}{du^{k}}e^{-\frac{u^{2}}{2}%
},\quad k=0,1,\ldots .
\end{equation*}

In particular, for a zero-mean Gaussian random field $Y,$ for $k\geq 1,$
\begin{equation}
\mathrm{E}\ H_{k}(Y(\mathbf{x}))=0,\quad \mathrm{E}\ \left( H_{k}(Y(\mathbf{x%
}))\ H_{m}(Y(\mathbf{y}))\right) =\delta _{m,k}\ m!\ \left( \mathrm{E}[Y(%
\mathbf{x})Y(\mathbf{y})]\right) ^{m}  \label{2.7}
\end{equation}%
\noindent (see, for example, \cite{Peccati}).

We use some ideas from  \cite{Ivanov},
Sections 1.4, 1.5 and 2.1). Consider the uniform distribution on
$D(T)$ with the density:
\begin{equation}
P_{D(T)}(\mathbf{x})= T^{-d}|D|^{-1}\mathbb{I}_{\mathbf{x}\in
D(T)},\quad \mathbf{x}\in \mathbb{R}^{d},\label{equd}
\end{equation}
\noindent where $\mathbb{I}_{\mathbf{x}\in D(T)}$ denotes  the
indicator function of set $D(T).$

Let $\mathbf{U}$ and $\mathbf{V}$ be two  independent and uniformly
distributed  inside the set $D(T)$ random vectors. We denote
$\psi_{D(T)}(\rho),$  the density of the Euclidean distance
$\|\mathbf{U}-\mathbf{V}\|.$ Note that $\psi_{D(T)}(\rho)=0,$ if
$\rho>\mbox{diam}\left( D(T)\right),$ and $\psi_{D(1)}(\rho)$ is
bounded, where $\mbox{diam}\left( D(T)\right)$ is the diameter of
the set $D(T).$

Using the above notation, we obtain
\begin{eqnarray}
\int_{D(T)}\int_{D(T)}G(\|\mathbf{x}-\mathbf{y}\|)d\mathbf{x}d\mathbf{y}&=&
|D(T)|^{2}E\left[G\left(
\|\mathbf{U}-\mathbf{V}\|\right)\right]\nonumber\\
&=& |D|^{2}T^{2d}\int_{0}^{\mbox{diam}(D(T))}
G(\rho)\psi_{D(T)}(\rho)d\rho,
 \label{intG}
\end{eqnarray}

\noindent for any Borel function $G$ such that the Lebesgue integral
(\ref{intG}) exists. In particular, under {\bfseries Conditions
A1--A2} for $0<\alpha <d/2,$ and $T\rightarrow \infty,$ we obtain
\begin{eqnarray}
\sigma^{2}(T)&=&\mbox{Var}\left[\int_{D(T)}H_{2}(Y(\mathbf{x}))d\mathbf{x}\right]=
2\int_{D(T)}\int_{D(T)}\frac{\mathcal{L}^{2}\left(\|\mathbf{x}-\mathbf{y}\|\right)}{\|\mathbf{x}-\mathbf{y}\|^{2\alpha
}}d\mathbf{x}d\mathbf{y}\nonumber\\
&=& 2!
|D|^{2}T^{2d}\int_{0}^{\mbox{diam}(D(T))}\mathcal{L}^{2}(\rho)\rho^{-2\alpha
}\psi_{D(T)}(\rho )d\rho.\label{eqsigma}
\end{eqnarray}

 In equation (\ref{eqsigma}), consider the change of
variable $u=\rho/T.$ Applying the consistency of the uniform
distribution with a homothetic transformation, and  the asymptotic
properties of slowly varying functions (see Theorem 2.7 in \cite{Seneta}) we get
\begin{eqnarray}
\sigma^{2}(T)&=& 2|D|^{2}T^{2d-2\alpha
}\int_{0}^{\mbox{diam}(D)}u^{-2\alpha
}\mathcal{L}^{2}(uT)\psi_{D}(u)du\nonumber\\
&=& |D|^{2}T^{2d-2\alpha }\mathcal{L}^{2}(T)[a_{d}(D)]^{2}(1+o(1)),\
0<\alpha<d/2,\quad  T\rightarrow \infty,\label{eqintsigma}
\end{eqnarray}
\noindent where, by (\ref{intG}),
\begin{equation}a_{d}(D)=\left[2\int_{0}^{\mbox{diam}(D)}u^{-2\alpha
}\psi_{D}(u)du\right]^{1/2}=\left[2\int_{D}\int_{D}\frac{d\mathbf{x}d\mathbf{y}}{\|\mathbf{x}-\mathbf{y}\|^{2\alpha
}}\right]^{1/2}.\label{functad}
\end{equation}
\noindent More details,  including properties of slowly varying
functions, can be found in \cite{Anh}.

If $D$ is the ball $\mathcal{B}_{T}(\mathbf{0})=\{ \mathbf{x}\in
\mathbb{R}^{d}:\ \|\mathbf{x}\|\leq T\},$ then (see \cite{Ivanov}, Lemma 1.4.2)
\begin{equation}\psi_{\mathcal{B}_{T}(\mathbf{0})}(\rho)=T^{-d}I_{1-\left(\frac{\rho}{2T}\right)^{2}}\left(\frac{d+1}{2},\frac{1}{2}\right)d\rho,\quad
0\leq \rho\leq 2T,\label{densityball}\end{equation} \noindent where
\begin{equation}
I_{\mu}(p,q)= \frac{\Gamma (p+q)}{\Gamma(p)\Gamma
(q)}\int_{0}^{\mu}t^{p-1}(1-t)^{q-1}dt,\quad \mu\in [0,1],\quad p>0,\
q>0,\label{funcball}
\end{equation}
\noindent is the incomplete beta function. In this case, one can
show (see Lemma 2.1.3 in \cite{Ivanov})
\begin{equation}
a_{d}(B_{1}(\mathbf{0}))=\frac{2^{d-2\alpha +2}\pi^{d-1/2}\Gamma
\left(\frac{d-2\alpha +1}{2}\right)}{(d-2\alpha )\Gamma
\left(\frac{d}{2}\right)\Gamma (d-\alpha +1)}.\label{eedeball}
\end{equation}

For $d=1,$ $D=[0,1],$
$$a_{1}([0,1])=2\int_{0}^{1}\int_{0}^{1}\frac{dxdy}{|x-y|^{2\alpha
}}=\frac{1}{(1-\alpha )(1-2\alpha )},\quad 0<\alpha<1/2.$$

Let us now consider domain $D=[0,l_{1}]\times \dots \times
[0,l_{d}]\subset \mathbb{R}^{d},$ $l_{i}>0,$ $i=1,\dots,d.$ The
Dirichlet negative Laplacian operator on such a domain has
eigenvectors $\{\phi_{k}\}_{k\geq 1}$ and eigenvalues
$\{\lambda_{k}(-\Delta)\}_{k\geq 1}$ given by
\begin{eqnarray}
\phi_{k}(\mathbf{x}) &=& \prod_{i=1}^{d}\sin\left(\frac{\pi k_{i}
x_{i}}{l_{i}}\right),\quad \mathbf{x}=(x_{1},\dots,x_{d})\in
[0,l_{1}]\times \dots \times[0,l_{d}],\quad k_{i}\geq 1,\quad
i=1,\dots,d\nonumber\\
\lambda_{\mathbf{k}}(-\Delta ) &=&
\sum_{i=1}^{d}\frac{\pi^{2}k_{i}^{2}}{l_{i}^{2}},\quad
\mathbf{k}=(k_{1},\dots,k_{d}),\quad k_{i}\geq 1,\quad
i=1,\dots,d.\nonumber\\ \label{eqdomainrect}
\end{eqnarray}

The norm $\|\mathcal{K}_{\alpha
}^{2}\|_{\mathcal{L}(L^{2}(D))}$ of $\mathcal{K}_{\alpha }^{2}$ in $\mathcal{L}(L^{2}(D)),$ the
space of bounded linear operators on $L^{2}(D),$ is given by the
supremum of its eigenvalues. From Theorem 4.5(ii) in
\cite{ChengSong05},
\begin{equation}\lambda_{k}(\mathcal{K}_{\alpha }^{2})\leq 2
\left[\sum_{i=1}^{d}\frac{\pi^{2}k_{i}^{2}}{l_{i}^{2}}\right]^{\alpha
-d},\quad d-\alpha \in (0,1).\label{exth2i}
\end{equation}

\noindent Note that Theorem 4.5(ii) in \cite{ChengSong05} holds for
a bounded convex domain in $\mathbb{R}^{d}.$ From equation
(\ref{exth2i}), if $\min \{ l_{1},\dots, l_{d}\}\leq
\sqrt{\frac{d}{2}}\pi,$
\begin{equation}\|\mathcal{K}_{\alpha
}^{2}\|_{\mathcal{L}(L^{2}(D))}\leq 1.\label{nc}\end{equation}

\begin{theorem}
\label{pr1} Let $D$ be a regular bounded  open  domain. Assume that
{\bfseries Conditions A1} and {\bfseries A2} are satisfied. The
following assertions then hold:
\medskip

\begin{itemize}
\item[(i)] As $T\longrightarrow \infty ,$  the functional $S_{T}$ in (\ref{eq2}) converges in distribution
sense  to a zero-mean  random variable $S_{\infty } ,$ with
characteristic function given by
\begin{equation}
\psi (z)=E\left[\exp(izS_{\infty})\right]=\exp \left( \frac{1}{2}\sum_{m=2}^{\infty }\frac{\left( 2\mathrm{i}%
z\right) ^{m}}{m}c_{m}\right),\quad z\in \mathbb{R},  \label{chf}
\end{equation}%
\noindent where, for $m\geq 2,$
\begin{eqnarray}
c_{m}&=&\mbox{Tr}\left(\mathcal{K}_{\alpha }^{m}\right) \nonumber\\
&=&\int_{D}\underset{(m)}{\cdots }\int_{D}\frac{1}{\Vert \mathbf{x}_{1}-%
\mathbf{x}_{2}\Vert ^{\alpha }}\frac{1}{\Vert \mathbf{x}_{2}-\mathbf{x}%
_{3}\Vert ^{\alpha }}\cdots \frac{1}{\Vert \mathbf{x}_{m}-\mathbf{x}%
_{1}\Vert ^{\alpha }}d\mathbf{x}_{1}\dots d\mathbf{x}_{m}.
\label{eqcoefchfbcc}
\end{eqnarray}

\item[(ii)] The functional
\begin{equation*}
S_{T}^{H}=\frac{1}{\mathcal{L}(T)T^{d-\alpha }}\left[\int_{D (T)}G(Y(\mathbf{%
x}))d\mathbf{x}-C_{0}^{H}T^{d}|D|\right]
\end{equation*}
\noindent converges in distribution sense, as $T\longrightarrow \infty,$ to
the random variable $\frac{1}{2}C_{2}^{H}S_{\infty },$ with $S_{\infty }$
having characteristic function (\ref{chf}), and with $G\in L^{2}(%
\mathbb{R},\varphi (x)dx)$ having Hermite rank $m=2.$ Here,
\begin{eqnarray}
C_{0}^{H} &=&\int_{\mathbb{R}}
G(u)H_{0}(u)\varphi(u)du=\mathrm{E}[G(Y(\mathbf{x}))]
\notag \\
C_{2}^{H}&=&\int_{\mathbb{R}}G(u)H_{2}(u)\varphi (u)du,  \notag
\end{eqnarray}
\noindent respectively denote the $0$th and $2$th Hermite coefficients of
the function $G.$
\end{itemize}
\end{theorem}
\begin{remark}
\textrm{Note that {\bfseries Condition A2} is satisfied by the slowly varying function (\ref{svfex}) with $C=1,$ for}
$D=\mathcal{B}_{1}(\mathbf{0})=\{\mathbf{x}: \|\mathbf{x}\|\leq
1\}.$
\end{remark}
\begin{proof}
We first prove (i). Since $EY^{2}(\mathbf{x})=1,$ $$\int_{D
(T)}d\mathbf{x}=\int_{D
(T)}E\left[Y^{2}(\mathbf{x})\right]d\mathbf{x}=E\left[\int_{D(T)}Y^{2}(\mathbf{x})d\mathbf{x}\right]=\sum_{j=1}^{\infty
}\lambda_{j,T}(R_{Y,D(T)})E\eta_{j}^{2}= \sum_{j=1}^{\infty
}\lambda_{j,T}(R_{Y,D(T)}).$$
 \noindent From
Definition \ref{def1},
 Lemma \ref{lemm0}, and Remark \ref{rem2},
 one has
\begin{eqnarray}
\psi_{T}(z)&=& E\left[\exp\left( \frac{\mathrm{i}z}{d_{T}}\int_{D (T)}(Y^{2}(\mathbf{x})-1)d\mathbf{x}\right)\right]\nonumber\\
&=& \exp\left(-\frac{\mathrm{i}z\sum_{j=1}^{\infty }\lambda_{j,T}(R_{Y,D(T)})}{d_{T}}\right)\prod_{j=1}^{\infty
}\left(1-2\mathrm{i}z\frac{\lambda_{j,T}(R_{Y,D(T)})}{d_{T}}\right)^{-1/2}\nonumber\\
&=&  \exp\left(-\frac{\mathrm{i}z\sum_{j=1}^{\infty
}\lambda_{j,T}(R_{Y,D(T)})}{d_{T}}\right)\left[\mathcal{D}_{T}\left(\frac{2\mathrm{i}z}{d_{T}}\right)\right]^{-1/2} \nonumber\\
 &=&\exp\left(-\frac{\mathrm{i}z\sum_{j=1}^{\infty
}\lambda_{j,T}(R_{Y,D(T)})}{d_{T}}\right) \exp\left(\frac{1}{2}
\sum_{m=1}^{\infty
}\frac{1}{m}\left(\frac{2\mathrm{i}z}{d_{T}}\right)^{m}\mbox{Tr}\left(
R_{Y,D (T)}^{m}\right)\right)\nonumber
\end{eqnarray}
\begin{eqnarray}
 &=&
 \exp\left(-\frac{\mathrm{i}z\sum_{j=1}^{\infty }\lambda_{j,T}(R_{Y,D(T)})}{d_{T}}+\frac{\mathrm{i}z\sum_{j=1}^{\infty
}\lambda_{j,T}(R_{Y,D(T)})}{d_{T}}\right. \nonumber\\
& &\left. \hspace*{1cm} +\frac{1}{2} \sum_{m=2}^{\infty
}\frac{1}{m}\left(\frac{2\mathrm{i}z}{d_{T}}\right)^{m}\mbox{Tr}\left(
R_{Y,D (T)}^{m}\right)\right)\nonumber\\
&=&\exp\left(\frac{1}{2} \sum_{m=2}^{\infty }\frac{1}{m}\left(\frac{2\mathrm{i}z}{d_{T}}\right)^{m}\mbox{Tr}\left(
R_{Y,D (T)}^{m}\right)\right).\nonumber\\
\label{chfseqtt}
\end{eqnarray}

Furthermore, under \textbf{Condition A2}, there exists a positive
constant $C$ such that
\begin{eqnarray}
\frac{1}{d_{T}^{2}}\mbox{Tr}\left(
R_{Y,D(T)}^{2}\right)&=&\int_{D}\int_{D}\frac{\mathcal{L}(T\|\mathbf{x}_{1}-\mathbf{x}_{2}\|)}{\mathcal{L}(T)}\frac{\mathcal{L}(T\|\mathbf{x}_{2}-\mathbf{x}_{1}\|)}{\mathcal{L}(T)}
\frac{1}{\|\mathbf{x}_{1}-\mathbf{x}_{2}\|^{2\alpha
}}d\mathbf{x}_{1}d\mathbf{x}_{2}\nonumber\\
&\leq &
C\int_{D}\int_{D}\frac{1}{\|\mathbf{x}_{1}-\mathbf{x}_{2}\|^{2\alpha
}}d\mathbf{x}_{1}d\mathbf{x}_{2}=C \mbox{Tr}\left(
\mathcal{K}_{\alpha }^{2}\right)<\infty
 \label{eqm2}
 \end{eqnarray}
 \begin{eqnarray}
& &\frac{1}{d_{T}^{m}}
 \mbox{Tr}\left( R_{Y,D(T)}^{m}\right)=\nonumber\\ & & =\frac{1}{[\mathcal{L}(T)]^{m}}\int_{D
}\underset{(m)}{\cdots }\int_{D
}\frac{\mathcal{L}(T\|\mathbf{x}_{1}-\mathbf{x}_{2}\|)}{\|\mathbf{x}_{1}-\mathbf{x}_{2}\|^{\alpha
}}\frac{\mathcal{L}(T\|\mathbf{x}_{2}-\mathbf{x}_{3}\|)}{\|\mathbf{x}_{2}-\mathbf{x}_{3}\|^{\alpha
}}\cdots
\frac{\mathcal{L}(T\|\mathbf{x}_{m}-\mathbf{x}_{1}\|)}{\|\mathbf{x}_{m}-\mathbf{x}_{1}\|^{\alpha
}}d\mathbf{x}_{1}\dots
d\mathbf{x}_{m}\nonumber\\
& & \leq C   \int_{D }\underset{(m)}{\cdots }\int_{D
}\frac{1}{\|\mathbf{x}_{1}-\mathbf{x}_{2}\|^{\alpha
}}\frac{1}{\|\mathbf{x}_{2}-\mathbf{x}_{3}\|^{\alpha
}}\cdots \frac{1}{\|\mathbf{x}_{m}-\mathbf{x}_{1}\|^{\alpha }}d\mathbf{x}_{1}\dots d\mathbf{x}_{m} \nonumber\\
& & =C \mbox{Tr}\left( \mathcal{K}_{\alpha }^{m}
\right)<\infty,\quad  m>2,
 \label{chcoeff}
\end{eqnarray}
\noindent since $\| \mathcal{K}_{\alpha }^{m}\|_{1}\leq M\|
\mathcal{K}_{\alpha }^{2}\|_{1},$ for $m>2$ and $M>0.$

 From equations (\ref{chfseqtt})--(\ref{chcoeff}), for every $T>0,$
\begin{eqnarray}
 |\psi_{T}(z)|&=& \left|\exp\left(\frac{1}{2} \sum_{m=2}^{\infty }\frac{(-1)^{m-1}}{2m-2}\left( \frac{2z}{d_{T}}\right)^{2m-2}\mbox{Tr}\left(
R_{Y,D(T)}^{2m-2}\right)\right)\right|\nonumber\\
&\times & \left|\exp\left(\frac{i}{2} \sum_{m=3}^{\infty }
\frac{(-1)^{m}}{2m-3}\left( \frac{2z}{d_{T}}\right)^{2m-3}\mbox{Tr}\left(
R_{Y,D(T)}^{2m-3}\right)\right)\right|\label{ereim}\\
&= & \left|\exp\left(\frac{1}{2} \sum_{m=2}^{\infty
}\frac{(-1)^{m-1}}{2m-2}\left( \frac{2z}{d_{T}}\right)^{2m-2}\mbox{Tr}\left(
R_{Y,D(T)}^{2m-2}\right)\right)\right|\nonumber\\
&\times & \left[\cos^{2} \left(\frac{1}{2}\sum_{m=3}^{\infty }
\frac{(-1)^{m}}{2m-3}\left( \frac{2z}{d_{T}}\right)^{2m-3}\mbox{Tr}\left(
R_{Y,D(T)}^{2m-3}\right)\right)\right.\nonumber\\
& & \left.+\sin^{2}\left(\frac{1}{2}\sum_{m=3}^{\infty }
\frac{(-1)^{m}}{2m-3}\left(
\frac{2z}{d_{T}}\right)^{2m-3}\mbox{Tr}\left(
R_{Y,D(T)}^{2m-3}\right)\right)\right]^{1/2} \nonumber
\\
&=& \left|\exp\left(\frac{-1}{2} \sum_{m=2}^{\infty
}\frac{(-1)^{m}}{2m-2}\left(
\frac{2z}{d_{T}}\right)^{2m-2}\mbox{Tr}\left(
R_{Y,D(T)}^{2m-2}\right)\right)\right|\label{eq0desarrollo}\\
&=& \left|\exp\left(\frac{-1}{2} \sum_{n=1}^{\infty
}\frac{(-1)^{2n}}{4n-2}\left(
\frac{2z}{d_{T}}\right)^{4n-2}\mbox{Tr}\left(
R_{Y,D(T)}^{4n-2}\right)\right)\right|\nonumber
\\
&\times &\left|\exp\left(\frac{-1}{2} \sum_{n=1}^{\infty
}\frac{(-1)^{2n+1}}{4n}\left(
\frac{2z}{d_{T}}\right)^{4n}\mbox{Tr}\left(
R_{Y,D(T)}^{4n}\right)\right)\right|\label{eq1desarrollo}
\end{eqnarray}
\begin{eqnarray}
&=&\left|\exp\left(\frac{-1}{2} \sum_{n=1}^{\infty
}\frac{1}{4n-2}\left( \frac{2z}{d_{T}}\right)^{4n-2}\mbox{Tr}\left(
R_{Y,D(T)}^{4n-2}\right)\right)\right|\nonumber\\
&\times &\left|\exp\left(\frac{1}{2} \sum_{n=1}^{\infty
}\frac{1}{4n}\left( \frac{2z}{d_{T}}\right)^{4n}\mbox{Tr}\left(
R_{Y,D(T)}^{4n}\right)\right)\right|\nonumber
\\
&\leq &\left|\exp\left(\frac{-1}{2} \sum_{n=1}^{\infty
}\frac{1}{4n-2}\left( \frac{2z}{d_{T}}\right)^{4n-2}\mbox{Tr}\left(
R_{Y,D(T)}^{4n-2}\right)\right)\right|\nonumber\\
&\times &\left|\exp\left(\frac{C}{2} \sum_{n=1}^{\infty
}\frac{1}{4n}\left( 2z\right)^{4n}\mbox{Tr}\left(
\mathcal{K}_{\alpha }^{4n}\right)\right)\right|
\nonumber\\
&\leq &\left|\exp\left(\frac{C}{2} \sum_{n=1}^{\infty
}\frac{1}{4n}\left( 2z\right)^{4n}\mbox{Tr}\left(
\mathcal{K}_{\alpha }^{4n}\right)\right)\right|\nonumber\\
&= &\left|\exp\left(\frac{C}{8} \sum_{n=1}^{\infty
}\frac{1}{n}\left( 16z^{4}\right)^{n}\mbox{Tr}\left(
(\mathcal{K}_{\alpha
}^{4})^{n}\right)\right)\right|=\left[\mathcal{D}_{\mathcal{K}_{\alpha
}^{4}}(16z^{4})\right]^{-C/8} <\infty,
 \label{eqdctnn}
 \end{eqnarray}
\noindent where we have applied, inside the argument of the exponential,  the straightforward identities $i^{2m-2}=(i^{2})^{m-1}=(-1)^{m-1},$ $|\exp(iu)|=\cos^{2}(u)+\sin^{2}(u)=1,$ and the fact that
the sequence of natural numbers $m=2,3,4,5,6\dots=\mathbb{N}-\{0,1\}$ can be obtained as the union of the sequences $\{2m-2\}_{m\geq 2}=2,4,6,\dots $ and $\{2m-3\}_{m\geq 3}=3,5,7,\dots .$ Hence, in
the above equation, the sum in $\mathbb{N}-\{0,1\}$ can be splitted into  the sums  $\sum_{m=2}^{\infty }(-1)^{m-1}f(2m-2)$ and $\sum_{m=3}^{\infty }(-1)^{m}f(2m-3).$
 Moreover, in  (\ref{eq1desarrollo}), we consider the sequence $\{2m-2\}_{m\geq 2}=2,4,6,8,10,12,\dots $ as the union of the sequences $\{4n-2\}_{n\geq 1}=2,6,10,\dots,$ and
 $\{4n\}_{n\geq 1}=4,8,12,\dots,$ corresponding to the changes of variable $m=2n$
and $m=2n+1.$ Thus, for $m=2n,$ $2m-2=4n-2,$ and, for $m=2n+1,$
$2m-2=4n+2-2=4n.$ The sum $\sum_{m=2}^{\infty }(-1)^{m}f(2m-2)$ can
then be splitted into the two sums
 $\sum_{n=1}^{\infty }(-1)^{2n}f(4n-2)$ and  $\sum_{n=1}^{\infty }(-1)^{2n+1}f(4n).$
 Furthermore, the last identity in (\ref{eqdctnn}) is obtained
from the Fredholm determinant formula  \begin{equation}\mathcal{D}_{\mathcal{K}_{\alpha
}^{4}}=\mbox{det}(I-\omega \mathcal{K}_{\alpha
}^{4})=\exp\left(-\sum_{k=1}^{\infty }%
\frac{\mbox{Tr}[\mathcal{K}_{\alpha
}^{4}]^{k}}{k}\omega^{k}\right)=\exp\left(-\sum_{k=1}^{\infty
}\sum_{l=1}^{\infty}[\lambda_{l}(\mathcal{K}_{\alpha
}^{4})]^{k}\frac{\omega^{k}}{k}\right)\label{fdfbb}
\end{equation} \noindent of $\mathcal{K}_{\alpha }^{4}$ at point $\omega =16z^{4},$
which is finite for $|\omega|<\frac{1}{\|\mathcal{K}_{\alpha
}^{4}\|_{1}},$ i.e.,  for $|z|<\frac{1}{4\|\mathcal{K}_{\alpha
}^{4}\|_{1}^{1/4}},$ since, from Theorem \ref{WaeK} (see also
equations (\ref{RKD}) and  (\ref{functad})),
\begin{equation}
\mbox{Tr}\left( \mathcal{K}_{\alpha
}^{2}\right)=\int_{D}\int_{D}\frac{1}{\|\mathbf{x}-\mathbf{y}\|^{2\alpha
}}d\mathbf{x}d\mathbf{y}=\frac{[a_{d}(D)]^{2}}{2}<\infty.
\label{trk}
\end{equation}
\noindent Hence, $\|\mathcal{K}_{\alpha }^{4}\|_{1}\leq
\widetilde{M}\|\mathcal{K}_{\alpha
}^{2}\|_{1}=\widetilde{M}\mbox{Tr}\left(\mathcal{K}_{\alpha
}^{2}\right)<\infty ,$ for certain $\widetilde{M}>0.$

From (\ref{eqdctnn}), there exists $\widetilde{\psi}
(z)=\lim_{T\rightarrow \infty }|\psi_{T}(z)|<\infty,$ for $0<z <
\left[1/16\|\mathcal{K}_{\alpha }^{4}\|_{1}\right]^{1/4}.$ An
analytic continuation argument (see \cite{Lukacs}, Th. 7.1.1)
guarantees that $\widetilde{\psi} $ defines the unique limit
characteristic function for all real values of $z.$

From (\ref{eqdctnn}), we now prove that
$\widetilde{\psi}(z)=\psi(z),$ with $\psi(z)$ given in
 (\ref{chf})--(\ref{eqcoefchfbcc}).

Consider, for each $z\in \mathbb{R},$ the sequence
\begin{equation}\left\{
\frac{1}{m}\left(\frac{2iz}{d_{T}}\right)^{m}\mbox{Tr}\left(
R_{Y,D(T)}^{m}\right)\right\}_{m\geq 2},\label{split}
\end{equation}

\noindent involved in the series appearing in the argument of the
complex exponential defining  $\psi_{T}(z),$ in equation
(\ref{chfseqtt}). Keeping in mind that $(i)^{2m}=(-1)^{m},$ and that
$(i)^{2m+1}=(-1)^{m}i,$ let us split (\ref{split}) into the
subsequences:
\begin{eqnarray}
\{P_{m}\}_{m\geq 1}&=&
\left\{\frac{(-1)^{m}}{2m}\left(\frac{2z}{d_{T}}\right)^{2m}\mbox{Tr}\left(
R_{Y,D(T)}^{2m}\right)\right\}_{m\geq 1}\label{s1}\\
\{Q_{m}\}_{m\geq
1}&=&\left\{\frac{(-1)^{m}i}{2m+1}\left(\frac{2z}{d_{T}}\right)^{2m+1}\mbox{Tr}\left(
R_{Y,D(T)}^{2m+1}\right)\right\}_{m\geq 1}. \label{s2}
\end{eqnarray}
\noindent  Under \textbf{Condition A2}, the following two
inequalities hold:
\begin{eqnarray}
\left|\frac{(-1)^{m}}{2m}\left(\frac{2z}{d_{T}}\right)^{2m}\mbox{Tr}\left(
R_{Y,D(T)}^{2m}\right)\right|&\leq &
C\frac{1}{2m}\left(2|z|\right)^{2m}\mbox{Tr}\left(
\mathcal{K}_{\alpha }^{2m}\right)\label{fi}\\
\left|\frac{(-1)^{m}i}{2m+1}\left(\frac{2z}{d_{T}}\right)^{2m+1}\mbox{Tr}\left(
R_{Y,D(T)}^{2m+1}\right)\right|&\leq &
C\frac{1}{2m+1}(2|z|)^{2m+1}\mbox{Tr}\left( \mathcal{K}_{\alpha
}^{2m+1}\right). \label{fi2}
\end{eqnarray}

Moreover, for  $|z|<1/2,$
\begin{eqnarray}
\sum_{m=1}^{\infty}C\frac{1}{2m}\left(2|z|\right)^{2m}\mbox{Tr}\left(
\mathcal{K}_{\alpha }^{2m}\right)&\leq &
\sum_{m=1}^{\infty}\frac{C}{m}\left(2|z|\right)^{m}\mbox{Tr}\left(
\mathcal{K}_{\alpha }^{2m}\right) \nonumber\\
&=&\ln\left(\left[\mathcal{D}_{\mathcal{K}_{\alpha
}^{2}}(2|z|)\right]^{-C}\right)<\infty\label{fub}
\\
\sum_{m=1}^{\infty}\frac{C}{2m+1}(2|z|)^{2m+1}\mbox{Tr}\left(
\mathcal{K}_{\alpha }^{2m+1}\right)&\leq &
\sum_{m=1}^{\infty}\frac{C}{m}(2|z|)^{m}M\mbox{Tr}\left(
\mathcal{K}_{\alpha }^{2m}\right)
\nonumber\\
&=&\ln\left(\left[\mathcal{D}_{\mathcal{K}_{\alpha
}^{2}}(2|z|)\right]^{-CM}\right)<\infty, \label{fi3}
\end{eqnarray}
\noindent considering $|z|<\frac{1}{2}\wedge
\frac{1}{2\mbox{Tr}(\mathcal{K}_{\alpha }^{2})},$ where Definition
\ref{def1} of the Fredholm determinant of $\mathcal{K}_{\alpha
}^{2}$ has been applied, since, as commented before,
$\mathcal{K}_{\alpha }^{2}$ is in the trace class from Theorem
\ref{WaeK}. Here, $M$ is a positive constant such that, for every
$m\geq 1,$  $\mbox{Tr}\left( \mathcal{K}_{\alpha }^{2m+1}\right)\leq
M\mbox{Tr}\left( \mathcal{K}_{\alpha }^{2m}\right),$ in the case
where $\|\mathcal{K}_{\alpha }^{2}\|_{\mathcal{L}(L^{2}(D))}\leq 1.$
Otherwise, \begin{eqnarray} \mbox{Tr}\left( \mathcal{K}_{\alpha
}^{2m+1}\right)&\leq & M^{\prime }\mbox{Tr}\left(
\mathcal{K}_{\alpha }^{2m+2}\right)\nonumber\\
& \leq & M^{\prime }\left[\sup_{k\geq
1}\lambda_{k}(\mathcal{K}_{\alpha }^{2})\right]^{2} \mbox{Tr}\left(
\mathcal{K}_{\alpha }^{2m}\right).\nonumber \end{eqnarray} \noindent
Hence $M=M^{\prime}\left[\sup_{k\geq
1}\lambda_{k}(\mathcal{K}_{\alpha }^{2})\right]^{2}.$

Inequalities (\ref{fub}) and (\ref{fi3}) allow us to apply Dominated
Convergence Theorem for integration with respect to a counting
measure obtaining, for $|z|<\frac{1}{2}\wedge
\frac{1}{2\mbox{Tr}(\mathcal{K}_{\alpha }^{2})},$
\begin{eqnarray}
\lim_{T\rightarrow
\infty}\sum_{m=1}^{\infty}\frac{(-1)^{m}}{2m}\left(\frac{2z}{d_{T}}\right)^{2m}\mbox{Tr}\left(
R_{Y,D(T)}^{2m}\right)&=&\sum_{m=1}^{\infty}\lim_{T\rightarrow
\infty}\frac{(-1)^{m}}{2m}\left(\frac{2z}{d_{T}}\right)^{2m}\mbox{Tr}\left(
R_{Y,D(T)}^{2m}\right) \nonumber\\
&=&\sum_{m=1}^{\infty}\frac{(-1)^{m}}{2m}\left(2z\right)^{2m}\mbox{Tr}\left(
\mathcal{K}_{\alpha }^{2m}\right)\label{lim1}
\\
\lim_{T\rightarrow
\infty}\sum_{m=1}^{\infty}\frac{(-1)^{m}i}{2m+1}\left(\frac{2z}{d_{T}}\right)^{2m+1}\mbox{Tr}\left(
R_{Y,D(T)}^{2m+1}\right)&=&\sum_{m=1}^{\infty}\lim_{T\rightarrow
\infty}\frac{(-1)^{m}i}{2m+1}\left(\frac{2z}{d_{T}}\right)^{2m+1}\mbox{Tr}\left(
R_{Y,D(T)}^{2m+1}\right)\nonumber\\
&=&\sum_{m=1}^{\infty}\frac{(-1)^{m}i}{2m+1}\left(2z\right)^{2m+1}\mbox{Tr}\left(
\mathcal{K}_{\alpha }^{2m+1}\right)\label{lim2},
\end{eqnarray}
\noindent where we have considered that,  from (\ref{eq2b}), for
every $m\geq 2,$
\begin{equation} \lim_{T\rightarrow \infty}\frac{\mbox{Tr}\left(
R_{Y,D (T)}^{m}\right)}{d_{T}^{m}}=\mbox{Tr}\left(
\mathcal{K}_{\alpha }^{m} \right),\label{linpoinwise}
\end{equation}
\noindent which is finite from Theorem \ref{WaeK}, that ensures the
trace property of $\mathcal{K}_{\alpha }^{2}.$ Hence, for every
$m\geq 2,$ and $z,$ \begin{equation} \lim_{T\rightarrow
\infty}\frac{1}{m}\left(\frac{2\mathrm{i}z}{d_{T}}\right)^{m}\mbox{Tr}\left(
R_{Y,D
(T)}^{m}\right)=\frac{1}{m}\left(2\mathrm{i}z\right)^{m}\mbox{Tr}\left(
\mathcal{K}_{\alpha }^{m}
\right)=\frac{1}{m}\left(2\mathrm{i}z\right)^{m}c_{m},\label{linpoinwise2}
\end{equation}
\noindent with $c_{m}$ being given in equation (\ref{eqcoefchfbcc}).

From equations (\ref{chfseqtt}), (\ref{s1}), (\ref{s2}),
(\ref{lim1}),   (\ref{lim2}) and (\ref{linpoinwise2}), we obtain,
for $|z|<\frac{1}{2}\wedge \frac{1}{2\mbox{Tr}(\mathcal{K}_{\alpha
}^{2})},$

\begin{eqnarray}
\lim_{T\rightarrow \infty}\psi_{T}(z) &=& \lim_{T\rightarrow
\infty}\exp\left(\frac{1}{2} \sum_{m=2}^{\infty
}\frac{1}{m}\left(\frac{2\mathrm{i}z}{d_{T}}\right)^{m}\mbox{Tr}\left(
R_{Y,D (T)}^{m}\right)\right)\nonumber\\
&=& \lim_{T\rightarrow \infty}\exp\left(\frac{1}{2}
\sum_{m=1}^{\infty}P_{m}\right)\exp\left(\frac{1}{2}
\sum_{m=1}^{\infty}Q_{m}\right)\nonumber\\
&=&\exp\left(\frac{1}{2} \sum_{m=1}^{\infty}\lim_{T\rightarrow
\infty}P_{m}\right)\exp\left(\frac{1}{2} \sum_{m=1}^{\infty}
\lim_{T\rightarrow
\infty}Q_{m}\right)\nonumber\\
&=&\exp\left(\frac{1}{2}
\sum_{m=2}^{\infty}\frac{1}{m}\left(2\mathrm{i}z\right)^{m}\mbox{Tr}\left(\mathcal{K}_{\alpha
}^{m}\right)\right)=\psi (z).
 \label{limdef}
\end{eqnarray}
\noindent An analytic continuation argument (see \cite{Lukacs}, Th.
7.1.1) guarantees that $\psi $ defines the unique limit
characteristic function for all real values of $z.$

\bigskip

We now turn to the proof of (ii). Under  {\bfseries Condition \textbf{A1}}, since $B(\left\Vert \mathbf{x}\right\Vert )\leq 1,$ and $%
B(0)=1,$ we have
\[
B^{j}(\left\Vert \mathbf{x}\right\Vert )\leq B^{3}(\left\Vert
\mathbf{x}\right\Vert ),\quad j\geq 3.
\]%
Hence,
\[ K_T=
 \left[ \frac{1}{\mathcal{L}^{2}(T)T^{2d-2\alpha }}\right]E\left[\left(\int_{D
(T)}G(Y(\mathbf{x}))~d\mathbf{x}-C_{0}^{H}T^{d}\left\vert D
\right\vert -\frac{C_{2}^{H}}{2}\int_{D (T)}H_{2}(Y
(\mathbf{x}))~d\mathbf{x}\right)\right] ^{2}
\]%
\[
= \left[ \frac{1}{\mathcal{L}^{2}(T)T^{2d-2\alpha }}\right]
\sum_{j=3}^{\infty }\frac{(C_{j}^{H})^{2}}{j!}\int_{D (T)}\int_{D
(T)}B^{j}(\left\Vert \mathbf{x}-\mathbf{y}\right\Vert
)d\mathbf{x}d\mathbf{y}\leq
\]%
\begin{equation}
\leq \left[ \frac{1}{\mathcal{L}^{2}(T)T^{2d-2\alpha }}\right]
\int_{D (T)}\int_{D (T)}B^{3}(\left\Vert
\mathbf{x}-\mathbf{y}\right\Vert
)d\mathbf{x}d\mathbf{y}\left[\sum_{j=3}^{\infty
}\frac{(C_{j}^{H})^{2}}{j!}\right]. \label{e:KT}
\end{equation}
By {\bfseries Condition} \textbf{A1}, for any $\varepsilon >0,$ there exists $A_{0}>0,$ such that for$\left\Vert
\mathbf{x}-\mathbf{y}\right\Vert
>A_{0},$
$ B(\left\Vert \mathbf{x}-\mathbf{y}\right\Vert )<\varepsilon.$ Let
$\mathcal{D}_{1}=\{(\mathbf{x},\mathbf{y})\in D (T)\times
D(T):\left\Vert \mathbf{x}-\mathbf{y}\right\Vert \leq A_{0})\},\
\mathcal{D}_{2}=\{(\mathbf{x},\mathbf{y})\in D (T)\times
D(T):\left\Vert \mathbf{x}-\mathbf{y}\right\Vert
>A_{0})\},$

\begin{equation}
\int_{D (T)}\int_{D (T)}B^{3}(\left\Vert
\mathbf{x}-\mathbf{y}\right\Vert )d\mathbf{x}d\mathbf{y} =
\left\{\int \int_{\mathcal{D}_{1}}+\int
\int_{\mathcal{D}_{2}}\right\}B^{3}(\left\Vert
\mathbf{x}-\mathbf{y}\right\Vert )d\mathbf{x}d\mathbf{y} =
S_{T}^{(1)}+S_{T}^{(2)}.\label{eqids}
\end{equation}

Using the bound $B^{3}(\left\Vert \mathbf{x}-\mathbf{y}\right\Vert
)\leq 1$ on $\mathcal{D}_{1},$ and the bound $B^{3}(\left\Vert
\mathbf{x}-\mathbf{y}\right\Vert )<\epsilon B^{2}(\left\Vert
\mathbf{x}-\mathbf{y}\right\Vert )$ on $\mathcal{D}_{2},$ we obtain,
$$
\left| S_{T}^{(1)}\right| \leq \int \int_{\mathcal{D}_{1}}
 \left|B^{3}(\left\Vert  \mathbf{x}-\mathbf{y}\right\Vert
)\right|d\mathbf{x}d\mathbf{y} \leq M_{1} T^{d}$$ \noindent for a
suitable constant $M_{1}>0,$ and  for $T$ sufficiently large, under
\textbf{A2},

\begin{eqnarray}
\left\vert S_{T}^{(2)}\right\vert &\leq &\int
\int_{\mathcal{D}_{2}}\left\vert B^{3}(\left\Vert
\mathbf{x}-\mathbf{y}\right\Vert )\right\vert d\mathbf{x}d\mathbf{y}
\leq \epsilon  \int \int_{\mathcal{D}_{2}}B^{2}(\left\Vert
\mathbf{x}-\mathbf{y}\right\Vert )d\mathbf{x}d\mathbf{y} \nonumber\\
&\leq &\epsilon  \int_{D(T)}
\int_{D(T)}B^{2}(\|\mathbf{x}-\mathbf{y}\|)d\mathbf{x}d\mathbf{y}
=\epsilon \int_{D(T)}
\int_{D(T)}\frac{\mathcal{L}^{2}(\|\mathbf{x}-\mathbf{y}\|)}{\|\mathbf{x}-\mathbf{y}\|^{2\alpha
}}d\mathbf{x}d\mathbf{y}\nonumber\\
&=& \epsilon \mathcal{L}^{2}(T)T^{2d-2\alpha}\int_{D(1)}\int_{D(1)}
\frac{\mathcal{L}^{2}(T\|\mathbf{x}-\mathbf{y}\|)}{\|\mathbf{x}-\mathbf{y}\|^{2\alpha
}\mathcal{L}^{2}(T)}d\mathbf{x}d\mathbf{y} \nonumber\\ &\leq &
 \epsilon C\mathcal{L}^{2}(T)T^{2d-2\alpha}\int_{D(1)}\int_{D(1)}\frac{d\mathbf{x}d\mathbf{y}}{\|\mathbf{x}-\mathbf{y}\|^{2\alpha
}}<\infty, \quad 0<\alpha<d/2.\label{c2}
\end{eqnarray}
\noindent Thus, for
$$M_{2}=C\int_{D(1)}\int_{D(1)}\frac{d\mathbf{x}d\mathbf{y}}{\|\mathbf{x}-\mathbf{y}\|^{2\alpha
}},$$ \noindent  from (\ref{e:KT})--(\ref{c2}), we have
\begin{eqnarray}
K_{T} &\le&\left[ \frac{1}{\mathcal{L}(T)T^{d-\alpha }}\right]
^{2}\left[\sum_{j=3}^{\infty
}\frac{(C_{j}^{H})^{2}}{j!}\right]\int_{D (T)}\int_{D
(T)}B^{3}(\left\Vert \mathbf{x}-\mathbf{y}\right\Vert
)d\mathbf{x}d\mathbf{y}\nonumber\\
&\leq &
(M_{1}\vee M_{2})
\left[ \frac{T^{d}}{\mathcal{L}^{2}(T)T^{2d-2\alpha }}+\epsilon \frac{%
T^{2d-2\alpha }\mathcal{L}^{2}(T)}{\mathcal{L}^{2}(T)T^{2d-2\alpha
}}\right] \label{sidd}
 \end{eqnarray}
 \noindent is  arbitrarily small together with $\epsilon >0$ as $T\rightarrow \infty$.  The
desired result on weak-convergence then follows. 
\end{proof}
\begin{remark}
Note that the proof of Theorem \ref{pr1}(i) is slightly different
from the proof of Theorem 2 in \cite{RuizMedina}, for the case of
non-linear functionals of chi-squared random fields. Specifically,
we use here similar arguments to those ones considered in that
paper, to derive the existence of a limit. However, to obtain the
explicit limit characteristic function, we apply Dominated
Convergence Theorem in a different way, considering  equations
(\ref{split})--(\ref{s2}).
\end{remark}
\begin{remark}
\textrm{{Consider the case of $d=1$ and discrete time. That is, let $%
\{Y(t),\ t\in \mathbb{Z}\}$ be a stationary zero-mean Gaussian sequence with
unit variance and covariance function of the form
\begin{equation*}
B(t)=\frac{\mathcal{L}(t)}{|t|^{\alpha }},
\end{equation*}
\noindent for $0<\alpha <1/2.$ The proof of the weak convergence result in
\cite{Rosenblatt} and \cite{Taqqu75} is based on the following formula for the
characteristic function of a quadratic form of strong-correlated Gaussian
random variables:
\begin{eqnarray}
& & E\left[\exp\left\{\mathrm{i}z\frac{1}{d_{T}}\sum_{t=0}^{T-1}(Y^{2}(t)-1)\right\}%
\right]=\exp\left\{-\mathrm{i}zTd_{T}^{-1}\right\} \left[\mbox{det}%
\left(I_{T}-2\mathrm{i}zd_{T}^{-1}R_{T}\right)\right]^{-1/2}  \notag \\
& & =\exp\left\{ \sum_{k=2}^{\infty }(2\mathrm{i}zd_{T}^{-1})^{k}\frac{\mbox{Sp}%
R_{T}^{k}}{k}\right\},  \label{eqchwl}
\end{eqnarray}
\noindent where
\begin{equation}
\frac{1}{d_{T}^{k}}\mbox{Sp}R_{T}^{k}=\frac{1}{d_{T}^{k}}%
\sum_{i_{1}=0}^{T-1}\dots
\sum_{i_{k}=0}^{T-1}B(|i_{1}-i_{2}|)B(|i_{2}-i_{3}|)\dots B(|i_{k}-i_{1}|),
\label{eqcoef}
\end{equation}
}}

\noindent\textrm{with $d_{T}=T^{1-\alpha }\mathcal{L}(T),$ $R_{T}=E[Y\bar{Y}%
^{\prime }],$ $Y=(Y(0),\dots,Y(T-1))^{\prime },$ $\mbox{Sp}R_{T}$ denoting
the trace of the matrix $R_{T},$ and $I_{T}$ representing the identity
matrix of size $T$ (see p. 39 of  \cite{Mathai}). One
can get a direct extension of formulae (\ref{eqchwl}) and (\ref{eqcoef}) to
the stationary zero-mean Gaussian random process case in continuous time $%
\{Y(t),\ t\in \mathbb{R}\}$ (see  \cite{LeonenkoTaufer06}), but for
$d\geq 2 $ direct extensions of (\ref{eqchwl}) and (\ref{eqcoef})
are not available. The present paper addresses this problem by
applying alternative functional tools, like the Karhunen-Lo\`eve
expansion and Fredholm determinant formula, to overcome this
difficulty of discretization of the multidimensional parameter
space. Note that the Fredholm determinant formula appears in the
definition of the characteristic functional of  quadratic forms
defined in terms  of  Hilbert-valued zero-mean Gaussian random
variables (see, for example, Proposition 1.2.8 in \cite{Da Prato}).}
\end{remark}

\begin{remark}
\textrm{Expanding around zero the characteristic function
(\ref{chf}), we obtain the cumulants of random variable $S_{\infty
},$ that is, $\kappa_{1}=0,$ and}
\begin{equation}
\kappa_{k}=2^{k-1} (k-1)! c_{k},\quad k\geq 2,  \label{ecumulant}
\end{equation}
\noindent \textrm{where $c_{k}$ are defined as in equation
(\ref{eqcoefchfbcc}).  The derivation of explicit expressions for
$c_{k}$ would lead to the computation of the moments or cumulants of
the limit distribution. This aspect will constitute the subject of a
subsequent paper.}
\end{remark}

\section{Infinite series representation and eigenvalues}

\label{proprd} The representation of the Rosenblatt-type
distribution  as the sum of an infinite series of weighted
independent chi-squared random variables is derived in this section.
As in the classical case (see Proposition 2 in \cite{Dobrushin}), this series expansion is obtained from the double
Wiener-It{\^{o}} stochastic integral representation of $S_{\infty }$
in the spectral domain (see Theorem \ref{par}). Proposition
\ref{twooperators} and Corollary \ref{lfes} below establish the
connection between the eigenvalues of operator $\mathcal{K}_{\alpha
}$ in (\ref{RKD}) and the weights appearing in the series
representation derived.

\medskip

\noindent The following condition will be required for the
derivation of  Theorem \ref{par}(ii) below.

\medskip

\noindent {\bfseries Condition A3.} Suppose that {\bfseries Condition A1} holds, and there exists a spectral density $f_{0}(\|\boldsymbol{\lambda }\|),$
$\boldsymbol{\lambda }\in \mathbb{R}^{d},$ being decreasing function for $\|\boldsymbol{\lambda }\|\in (0,\varepsilon ],$   with $\varepsilon >0.$

\medskip

If {\bfseries Condition A3} holds,   from equation (\ref{cov}),
applying a Tauberian Theorem (see \cite{Doukhan},
and Theorems 4 and 11 in \cite{LeonenkoOlenko14}),
 \begin{equation}f_{0}(\|\boldsymbol{\lambda }\|)\sim c(d,\alpha)\mathcal{L}\left( \frac{1}{\|\boldsymbol{\lambda }\|}\right)\|\boldsymbol{\lambda }\|^{\alpha -d},\quad 0<\alpha <d,
 \quad  \|\boldsymbol{\lambda }\|\rightarrow 0.\label{asympsd}
\end{equation} Here,
 $c(d,\alpha )= \frac{\Gamma \left(\frac{d-\alpha }{2}\right)}{2^{\alpha }\pi^{d/2}\Gamma \left(\frac{\alpha }{2}\right)}$
 is defined in (\ref{eRc}).

 {\bfseries Condition A3}  holds, in particular,  for the correlation function (\ref{eqcorrfunctex}), with the isotropic spectral density
\begin{equation}f_{0}(\|\boldsymbol{\lambda }\|)=\frac{\|\boldsymbol{\lambda }\|^{1-\frac{d}{2}}}{2^{\frac{d}{2}-1}\pi^{\frac{d}{2}+1}}\int_{0}^{\infty}
K_{\frac{d}{2}-1}(\|\boldsymbol{\lambda}\|u)\frac{\sin\left(\gamma \ \mbox{arg}\left(1+u^{\beta }\exp\left(\frac{i\pi \beta }{2}\right)\right)\right)}{\left|1+u^{\beta }\exp\left(\frac{i\pi \beta }{2}\right)\right|^{\gamma }}u^{\frac{d}{2}}du,\label{asymp2ex}
\end{equation}
\noindent where $K_{\nu }(z)$ is the modified Bessel function of the
second kind. By Corollary 3.10 in \cite{Teo}, the spectral
density (\ref{asymp2ex}) satisfies (\ref{asympsd}), with $\alpha
=\beta \gamma <d.$

The zero-mean Gaussian random field $Y$ with an absolutely
continuous spectrum has the isonormal representation
\begin{equation}
Y(\mathbf{x})=\int_{\mathbb{R}^{d}}\exp\left(i\left\langle \boldsymbol{\lambda },\mathbf{x}\right\rangle \right)\sqrt{f_{0}(\|\boldsymbol{\lambda }\|)}
Z(d\boldsymbol{\lambda }),\label{isorep}
\end{equation}
\noindent where $Z$ is a  complex white noise Gaussian random
measure with Lebesgue control measure.

\bigskip

\begin{theorem}
\label{par} Let $D$ be a regular bounded open  domain.

\begin{itemize}
\item[(i)] For $0<\alpha <d/2,$ the following identities hold:
\begin{eqnarray}
\int_{\mathbb{R}^{2d}}\left\vert K\left( \boldsymbol{\lambda }_{1}+%
\boldsymbol{\lambda }_{2},D\right) \right\vert ^{2}\frac{d\boldsymbol{%
\lambda }_{1}d\boldsymbol{\lambda }_{2}}{\left( \left\Vert \boldsymbol{%
\lambda }_{1}\right\Vert \left\Vert \boldsymbol{\lambda }_{2}\right\Vert \right) ^{d-\alpha
}}&=&\left[\frac{a_{d}\gamma(\alpha ) }{\sqrt{2}|D|}\right]^{2} = \frac{[\gamma(\alpha )]^{2}Tr(\mathcal{K}_{\alpha
}^{2})}{|D|^{2}}<\infty ,\nonumber\\
\label{6.7.1}
\end{eqnarray}%
\noindent where   $a_{d}$ is defined in (\ref{functad}), $\gamma
(\alpha )$ is introduced in equation (\ref{eRc}),
 and $K$ is the characteristic function
of the uniform distribution over set $D,$ given by
\begin{equation}
K\left( \boldsymbol{\lambda },D\right)
=\int_{D}e^{\mathrm{i}\left\langle \boldsymbol{\lambda
},\mathbf{x}\right\rangle }p_{D}\left( \mathbf{x}\right)
d\mathbf{x}=\frac{1}{\left\vert D\right\vert }\int_{D}e^{\mathrm{i}%
\left\langle \boldsymbol{\lambda },\mathbf{x}\right\rangle }d\mathbf{x}=%
\frac{\vartheta (\boldsymbol{\lambda })}{\left\vert D\right\vert },
\label{UD1}
\end{equation}%
\noindent with associated probability density function $p_{D}\left( \mathbf{x%
}\right) =1/\left\vert D\right\vert $ if $\mathbf{x}\in D,$ and $0$
otherwise.

\item[(ii)] Assume that {\bfseries Conditions A1, A2, A3} hold. Then, the random
variable $S_{\infty}$ admits the following double Wiener-It{\^o} stochastic
integral representation:
\begin{eqnarray}
S_{\infty} &=& |D|c\left( d,\alpha \right)
\int_{\mathbb{R}^{2d}}^{\prime
\prime } H(\boldsymbol{\lambda }_{1},\boldsymbol{\lambda }_{2}) \frac{%
Z\left( d\boldsymbol{\lambda } _{1}\right)Z\left( d\boldsymbol{\lambda }
_{2}\right) }{\left\Vert \boldsymbol{\lambda }_{1}\right\Vert ^{\frac{%
d-\alpha }{2}}\left\Vert \boldsymbol{\lambda }_{2}\right\Vert ^{\frac{%
d-\alpha }{2}}},  \label{6.7.3}
\end{eqnarray}
\noindent where $Z$ is a complex   white noise  Gaussian measure with Lebesgue control measure, and the notation $%
\int_{\mathbb{R}^{2d}}^{\prime \prime }$ means that one does not
integrate
on the hyperdiagonals $\boldsymbol{\lambda }_{1}=\pm \boldsymbol{\lambda }%
_{2}.$ Here, the kernel $H$ is given by:
\begin{equation}
H\left( \boldsymbol{\lambda }_{1} ,\boldsymbol{\lambda }_{2}\right)=K\left( \boldsymbol{\lambda }
_{1}+\boldsymbol{\lambda } _{2},D \right),  \label{eqH}
\end{equation}
\noindent and $c\left( d,\alpha \right) =\frac{\Gamma \left(\frac{d-\alpha}{2}\right)}{\pi^{d/2}2^{\alpha }\Gamma
(\alpha/2)}=\frac{1}{\gamma (\alpha )}.$

\end{itemize}
\end{theorem}
\begin{remark}
\textrm{Our goal in this paper is to focus of the case of Hermite
rank $m=2,$ which has very special properties not shared by the
higher orders, such as the existence of eigenvalues. We are aware of
the extension to all Hermite ranks, as described, for example, in
the more general and different approach presented in the   monograph
by Major (1981).}
\end{remark}

\begin{proof}
\noindent (i) From equation (\ref{eqrkdef}) and the proof of Theorem
\ref{WaeK},
\begin{eqnarray}
\|1_{D}\|_{\mathcal{H}_{2\alpha-d}}^{2}&=&
%(-\Delta )^{\alpha -d/2}(1_{D})(1_{D})=
\int_{D}\frac{1}{\gamma (d-2\alpha
)}\int_{D}\frac{1}{\|\mathbf{x}-\mathbf{y}\|^{2\alpha
}}d\mathbf{y}d\mathbf{x}\nonumber\\
&=& \frac{a_{d}^{2}}{2\gamma (d-2\alpha )}=\frac{1}{\gamma
(d-2\alpha )}\sum_{j=1}^{\infty }\lambda_{j}^{2}(\mathcal{K}_{\alpha
}^{2})=\frac{\mbox{Tr}(\mathcal{K}_{\alpha }^{2})}{\gamma (d-2\alpha
)} <\infty,\nonumber\\ \label{TRkalpha}
\end{eqnarray}
\noindent since $\mathcal{K}_{\alpha }^{2}$ is in the trace class.
Therefore, $1_{D}$ belongs to the Hilbert space
$\mathcal{H}_{2\alpha -d}$ with the inner product  introduced in
equation (\ref{ip}). From equation (\ref{eqrkdef}), we then obtain
$$\frac{a_{d}^{2}}{2\gamma (d-2\alpha
)}=\|1_{D}\|_{\mathcal{H}_{2\alpha -d}}^{2} = \frac{|D|^{2}}{(2\pi
)^{d}}\int_{\mathbb{R}^{d}}|K(\boldsymbol{\omega}_{1},D)|^{2}\|\boldsymbol{\omega}_{1}\|^{-(d-2\alpha
)}d\boldsymbol{\omega}_{1}.$$ \noindent It is well-known that the
Fourier transform defines an automorphism on the Schwartz space,
which, in particular, contains $C_{0}^{\infty}(D).$ Thus, the
Fourier transform of  any function in the space
$\mathcal{H}_{2\alpha -d}$ can be defined as the limit in the space
$\mathcal{H}_{2\alpha -d}$ of the Fourier transforms of functions in
$C_{0}^{\infty}(D).$ Therefore,  from equation (\ref{cfsf2}) with
$f(\mathbf{z})=|D|^{2}|K(\mathbf{z}, D)|^{2},$

\begin{eqnarray}
\frac{a_{d}^{2}}{2\gamma (d-2\alpha )}&=&\|1_{D}\|_{\mathcal{H}_{2\alpha -d}}^{2}= \frac{|D|^{2}}{(2\pi
)^{d}}\int_{\mathbb{R}^{d}}|K(\boldsymbol{\omega}_{1},D)|^{2}
\|\boldsymbol{\omega}_{1}\|^{-d+2\alpha }d\boldsymbol{\omega}_{1}\nonumber\\
&=&\frac{|D|^{2}}{(2\pi )^{d}}\frac{\gamma (2\alpha )}{[\gamma (\alpha
)]^{2}}\int_{\mathbb{R}^{d}}|K(\boldsymbol{\omega}_{1},D)|^{2}
\left[\int_{\mathbb{R}^{d}}\|\boldsymbol{\omega}_{1}-\boldsymbol{\omega}_{2}\|^{-d+\alpha
}\|\boldsymbol{\omega}_{2}\|^{-d+\alpha
}d\boldsymbol{\omega}_{2}\right]d\boldsymbol{\omega}_{1}\nonumber\\
 &=&\frac{|D|^{2}\gamma (2\alpha )}{(2\pi
)^{d}[\gamma (\alpha )]^{2}}\int_{\mathbb{R}^{2d}}|K(\boldsymbol{\lambda}_{1}+\boldsymbol{\lambda}_{2},D)|^{2}
\frac{d\boldsymbol{\lambda}_{1}d\boldsymbol{\lambda}_{2}}{(\|\boldsymbol{\lambda}_{1}\|\|\boldsymbol{\lambda}_{2}\|)^{d-\alpha
}}\nonumber.
  \label{eqsiv}
\end{eqnarray}
Hence, \begin{equation}\frac{a_{d}^{2}}{2} =\left[\frac{|D|}{\gamma (\alpha
)}\right]^{2}\int_{\mathbb{R}^{2d}}|K(\boldsymbol{\lambda}_{1}+\boldsymbol{\lambda}_{2},D)|^{2}
\frac{d\boldsymbol{\lambda}_{1}d\boldsymbol{\lambda}_{2}}{(\|\boldsymbol{\lambda}_{1}\|\|\boldsymbol{\lambda}_{2}\|)^{d-\alpha
}},\label{Tabconst}
\end{equation}

\noindent  since $ \frac{\gamma (2\alpha )\gamma (d-2\alpha )}{(2\pi
)^{d}}=1.$ Note that, we also have applied  the fact that, from
Remark \ref{insp},
$$1_{D}\star
1_{D}(\mathbf{x})=\int_{\mathbb{R}^{d}}1_{D}(\mathbf{y})1_{D}(\mathbf{x}+\mathbf{y})d\mathbf{y}=\int_{D}1_{D}(\mathbf{x}+\mathbf{y})d\mathbf{y}
\in L^{2}(D)\subseteq \mathcal{H}_{2\alpha -d},$$ \noindent since
$$\int_{\mathbb{R}^{d}}\left|\int_{\mathbb{R}^{d}}1_{D}(\mathbf{y})1_{D}(\mathbf{x}+\mathbf{y})d\mathbf{y}\right|^{2}d\mathbf{x}\leq
\left|\mathcal{B}_{R(D)}(\mathbf{0})\right|^{3},$$ \noindent where,
as before, $|\mathcal{B}_{R(D)}|$ denotes the Lebesgue measure of
the ball of center $\mathbf{0}$ and radius $R(D),$ with $R(D)$ being
equal to two times the diameter of the regular bounded open  set $D$
containing the point $\mathbf{0}.$ Hence, $\mathcal{F}(1_{D}\star
1_{D})(\boldsymbol{\lambda })=|D|^{2}|K(\boldsymbol{\lambda
},D)|^{2}$ belongs to the space of Fourier transforms of functions
in $\mathcal{H}_{2\alpha -d}.$ Summarizing, equation
(\ref{Tabconst}) provides the finiteness of (\ref{6.7.1}), i.e.,
assertion (i) holds due to the trace property of
$\mathcal{K}_{\alpha }^{2}$ for the regular bounded
  domain $D$ considered (see Theorem \ref{WaeK}).
%\vspace*{1cm}

\bigskip
\noindent   (ii) The proof of this part of Theorem \ref{par} can be
obtained as a  particular case of  Theorem 5 in \cite{LeonenkoOlenko14} (see also Remark 6 in that paper). Note that convexity is not
used in the proof of Theorem 5 of \cite{LeonenkoOlenko14}. An
outline of the proof of Theorem 5 in  \cite{LeonenkoOlenko14} for
the case of Hermite rank equal to two is now given.

Under {\bfseries Conditions A1, A3} (see also  (\ref{isorep})),
\begin{equation}Y(\mathbf{x})=\frac{|D (T)|}{(2\pi)^{d}}\int_{\mathbb{R}^{d}}\exp\left(\mathrm{i}\left\langle \mathbf{x}, \boldsymbol{\lambda }\right\rangle\right)
K\left( \boldsymbol{\lambda
},D(T)\right)f_{0}^{1/2}(\boldsymbol{\lambda
})Z(d\boldsymbol{\lambda }),\quad \mathbf{x}\in D(T).
\label{eqfunctIWSR}
\end{equation}
\noindent Using the self-similarity of
Gaussian white noise, and the It$\widehat{\mbox{o}}$ formula  (see, for example, \cite{Dobrushin}; \cite{Major}), we
obtain from equation (\ref{eqfunctIWSR})
\begin{eqnarray}
S_{T}&=&=\frac{1}{T^{d-\alpha }\mathcal{L}(T)}\int_{D
(T)}H_{2}(Y(\mathbf{x}))d\mathbf{x}\nonumber\\ &=&\frac{c\left( d,\alpha \right)|D (T)|}{T^{d-\alpha
}\mathcal{L}(T)}\int_{\mathbb{R}^{2d}}^{\prime \prime}K\left( \boldsymbol{\lambda } _{1}+%
\boldsymbol{\lambda } _{2},D(T)\right) \left( \frac{1}{c\left(
d,\alpha \right)}\prod_{j=1}^{2}f_{0}^{1/2}(\boldsymbol{\lambda
}_{j})\right) Z\left( d\boldsymbol{\lambda } _{1}\right) Z\left(
d\boldsymbol{\lambda }
_{2}\right)\nonumber\\
&\underset{d}{=}&\frac{c\left( d,\alpha \right)|D|}{T^{d-\alpha
}\mathcal{L}(T)}\int_{\mathbb{R}^{2d}}^{\prime \prime}K\left(
\boldsymbol{\lambda } _{1}+ \boldsymbol{\lambda } _{2},D
\right)\left( \frac{1}{c\left( d,\alpha
\right)}\prod_{j=1}^{2}f_{0}^{1/2}(\boldsymbol{\lambda
}_{j}/T)\right) Z\left( d\boldsymbol{\lambda } _{1}\right)
Z\left( d\boldsymbol{\lambda } _{2}\right).\nonumber\\
\label{eqifsr}
\end{eqnarray}
By the isometry property of multiple stochastic integrals
\begin{eqnarray}
& & \mathrm{E}\left[ S_{T}-c\left( d,\alpha \right)|D |
\int_{\mathbb{R}^{2d}}^{\prime \prime}
H(\boldsymbol{\lambda }_{1},\boldsymbol{\lambda }_{2}) \frac{%
Z\left( d\boldsymbol{\lambda } _{1}\right)Z\left(
d\boldsymbol{\lambda }
_{2}\right) }{\left\Vert \boldsymbol{\lambda }_{1}\right\Vert ^{\frac{%
d-\alpha }{2}}\left\Vert \boldsymbol{\lambda }_{2}\right\Vert ^{\frac{%
d-\alpha }{2}}} \right] ^{2}=\nonumber\\
& & =\int_{\mathbb{R}^{2d}}\left\vert K\left( \boldsymbol{\lambda }
_{1}+\boldsymbol{\lambda } _{2},D \right) \right\vert ^{2}[c\left(
d,\alpha \right)|D |]^{2} Q_{T}(\boldsymbol{\lambda
}_{1},\boldsymbol{\lambda }_{2})\frac{d\boldsymbol{\lambda }
_{1}d\boldsymbol{\lambda } _{2}}{\left\Vert \boldsymbol{\lambda }
_{1}\right\Vert ^{d-\alpha }\left\Vert \boldsymbol{\lambda }
_{2}\right\Vert ^{d-\alpha }}, \label{eq60}
\end{eqnarray}
\noindent where
\begin{equation}
Q_{T}(\boldsymbol{\lambda } _{1},\boldsymbol{\lambda } _{2})=\left(
\left[\frac{\left\Vert \boldsymbol{\lambda } _{1}\right\Vert
^{(d-\alpha )/2}\left\Vert \boldsymbol{\lambda } _{2}\right\Vert
^{(d-\alpha )/2}}{T^{d-\alpha }\mathcal{L}(T)c\left( d,\alpha
\right)}\prod_{j=1}^{2}f_{0}^{1/2}(\boldsymbol{\lambda }
_{j}/T)\right]-1\right)^{2}. \label{Qeq}
\end{equation}

From equation (\ref{asympsd}), under {\bfseries Condition A3},
 we obtain the pointwise
convergence of $Q_{T}(\boldsymbol{\lambda } _{1},\boldsymbol{\lambda
} _{2})$ to $0,$ as $T\rightarrow \infty.$ By Lebesgue's Dominated
Convergence Theorem, the integral converges to zero if there is some
integrable function which dominates integrands for all T. This fact
can be proved as in pp. 21--22 of \cite{LeonenkoOlenko14},
applying  previous assertion (i) derived in this theorem.
\end{proof}

Alternatively, in the proof  of Theorem \ref{par}(ii), the class $\widetilde{\mathcal{L}}%
\mathcal{C}$ of slowly varying functions, introduced in Definition 9
in \cite{LeonenkoOlenko13}, can also be considered. Note that  an
infinitely differentiable function $\mathcal{L}({\cdot})$ belongs to
the class $\widetilde{\mathcal{L}}\mathcal{C}$ if

\begin{itemize}
\item[1.] for any $\delta >0,$ there exists $\lambda_{0}(\delta )>0$ such
that $\lambda^{-\delta }\mathcal{L}(\lambda )$ is decreasing and $%
\lambda^{\delta }\mathcal{L}(\lambda )$ is increasing if $\lambda
>\lambda_{0}(\delta );$

\item[2.] $\mathcal{L}_{j}\in \mathcal{S}\mathcal{L},$ for all $j\geq 0,$
where $\mathcal{L}_{0}(\lambda ):=\mathcal{L},$
$\mathcal{L}_{j+1}(\lambda ):=\lambda \mathcal{L}_{j}^{\prime
}(\lambda ),$ with $\mathcal{S}\mathcal{L} $ being the class of
functions that are slowly varying at infinity and bounded on each
finite interval.
\end{itemize}

In that case,  the following lemma can be applied in the proof
of Theorem \ref{par}(ii).
\begin{lemma}
\label{LeonenkoOlenko12} \textit{Let $\alpha \in (0,d),$ $S\in
C^{\infty }(s_{n-1}(1)),$ and $\mathcal{L}\in
\widetilde{\mathcal{L}}\mathcal{C}.$ Let $\{\xi (\mathbf{x}),\
\mathbf{x}\in \mathbb{R}^{d}\}$ be a mean-square
continuous homogeneous random field with zero mean. Let the field $\xi (%
\mathbf{x})$ has the spectral density $f_{0}(\mathbf{u}),$
$\mathbf{u}\in \mathbb{R}^{d},$ which is infinitely differentiable
for all $\mathbf{u}\neq
0.$ If the covariance function $B(\mathbf{x}),$ $\mathbf{x}\in \mathbb{R}%
^{d},$ of the field has the following behavior }
\begin{itemize}
\item[(a)] $\|\mathbf{x}\|^{\alpha }B(\mathbf{x})\sim S\left( \frac{\mathbf{x%
}}{\|\mathbf{x}\|}\right)\mathcal{L}(\|\mathbf{x}\|),\quad \mathbf{x}%
\longrightarrow \infty ,$

\noindent \textit{the spectral density satisfies the condition}

\item[(b)] $\|\mathbf{u}\|^{d-\alpha }f_{0}(\mathbf{u})\sim \widetilde{S}%
_{\alpha ,d}\left( \frac{\mathbf{u}}{\|\mathbf{u}\|}\right)\mathcal{L}\left(%
\frac{1}{\|\mathbf{u}\|}\right),\quad \| \mathbf{u}\|\longrightarrow
0.$
\end{itemize}
\end{lemma}

On the other hand,  from Theorems \ref{WaeK}  and \ref{par}(i), the
spectral asymptotics of $\mathcal{K}_{\alpha }$ and the Dirichlet
Laplacian operator on $L^{2}(D)$ can be applied to verifying the
finiteness of (\ref{6.7.1}) for a wide class of  domains $D.$ Drum
and fractal drum are two families of well-known  regular compact
sets where Weyl's classical theorem on the asymptotic behavior of
the eigenvalues has been extended (see, for example, \cite{Gordon}; \cite{Lapidus}; \cite{Triebel97}). In particular, as
illustration of Theorem \ref{par}(i), we now refer to the case of
regular compact domains constructed from the finite union of convex
compact sets like balls, or by their difference which is the case,
for instance, of circular rings.

\subsubsection*{Examples}

Let $ D=\mathcal{B}_{1}(\mathbf{0})\cup
\mathcal{B}_{1}((2,0))\subset \mathbb{R}^{2}, $ with
$\mathcal{B}_{1}(\mathbf{0})=\{(x_{1},x_{2})\in \mathbb{R}^{2}:\ \sqrt{%
x_{1}^{2}+x_{2}^{2}}\leq 1\},$ and
$\mathcal{B}_{1}((2,0))=\{(x_{1},x_{2})\in \mathbb{R}^{2}:\
\sqrt{(x_{1}-2)^{2}+x_{2}^{2}}\leq 1\}.$ It is well-known
(see \cite{Ivanov}, p. 57, Lemma 2.1.3) that, for $\mathcal{B}_{1}(%
\mathbf{0})\subset \mathbb{R}^{2}$ and $0<\alpha <1,$
\begin{equation*}
Tr([\mathcal{K}_{\alpha }^{\mathcal{B}_{1}(\mathbf{0})}]^{2})=\int_{\mathcal{B}_{1}(%
\mathbf{0})}\int_{\mathcal{B}_{1}(\mathbf{0})}\frac{1}{\Vert \mathbf{x}-\mathbf{y%
}\Vert ^{2\alpha }}d\mathbf{y}d\mathbf{x}=\frac{2^{2-2\alpha +1}\pi ^{2-%
\frac{1}{2}}\Gamma (\frac{2-2\alpha +1}{2})}{(2-2\alpha )\Gamma (2-\alpha +1)%
},
\end{equation*}

\noindent where, to avoid confusion, for a subset $S,$ we have used the
notation $\mathcal{K}_{\alpha }^{S}$ to represent operator $\mathcal{K}%
_{\alpha }$ acting on the space $L^{2}(S),$ and $[\mathcal{K}_{\alpha
}^{S}]^{2}=\mathcal{K}_{\alpha }^{S}\mathcal{K}_{\alpha }^{S}.$

Hence,
\begin{eqnarray}
&&\int_{\mathcal{B}_{1}(\mathbf{0})\cup \mathcal{B}_{1}((2,0))}\int_{\mathcal{B}_{1}(%
\mathbf{0})\cup \mathcal{B}_{1}((2,0))}\frac{1}{\Vert
\mathbf{x}-\mathbf{y}\Vert
^{2\alpha }}d\mathbf{y}d\mathbf{x}  \notag  \label{eqillust1} \\
&\leq &\int_{\mathcal{B}_{3}(\mathbf{0})}\int_{\mathcal{B}_{3}(\mathbf{0})}%
\frac{1}{\Vert \mathbf{x}-\mathbf{y}\Vert ^{2\alpha }}d\mathbf{y}d\mathbf{x}
\notag \\
&=&\mbox{Tr}\left( [\mathcal{K}_{\alpha }^{\mathcal{B}_{3}(\mathbf{0}%
)}]^{2}\right) =3^{4-2\alpha }\mbox{Tr}\left( [\mathcal{K}_{\alpha }^{%
\mathcal{B}_{1}(\mathbf{0})}]^{2}\right) =\frac{3^{4-2\alpha
}2^{2-2\alpha +1}\pi ^{2-\frac{1}{2}}\Gamma (\frac{2-2\alpha
+1}{2})}{(2-2\alpha )\Gamma
(2-\alpha +1)}<\infty .  \notag \\
&&
\end{eqnarray}%
\noindent From Theorem \ref{par}(i), equation (\ref{eqillust1}) provides the
finiteness of (\ref{6.7.1}) for non-convex compact set $D=\mathcal{B}_{1}(%
\mathbf{0})\cup \mathcal{B}_{1}((2,0)).$

These computations can be easily
extended to the finite union of balls with the same or with different
radius, and to the case $d>2,$ considering the value of the integral
\begin{equation*}
\int_{\mathcal{B}_{R}(\mathbf{0})}\int_{\mathcal{B}_{R}(\mathbf{0})}\frac{1}{%
\Vert \mathbf{x}-\mathbf{y}\Vert ^{2\alpha }}d\mathbf{y}d\mathbf{x}%
=R^{2d-2\alpha }a_{d}(\mathcal{B}_{1}(\mathbf{0}))\frac{1}{2},
\end{equation*}%
\noindent where the constant $a_{d}(\mathcal{B}_{1}(\mathbf{0}))$ is defined in (\ref{eedeball}),
for $0<\alpha <d/2$ (see \cite{Ivanov}, p. 57, Lemma
2.1.3).

For the case of a circular ring, that is, for
$$
D= \mathcal{B}_{R_{1}}(%
\mathbf{0})\backslash \mathcal{B}_{R_{2}}(\mathbf{0})=\{ \mathbf{x}\in
\mathbb{R}^{2}:\ R_{2}<\|\mathbf{x}\|< R_{1}\},\quad R_{1}>R_{2}>0,
$$
 we can
proceed in a similar way to the above-considered example. Specifically,
\begin{eqnarray}
&&\int_{\mathcal{B}_{R_{1}}(\mathbf{0})\backslash \mathcal{B}_{R_{2}}(%
\mathbf{0})}\int_{\mathcal{B}_{R_{1}}(\mathbf{0})\backslash \mathcal{B}%
_{R_{2}}(\mathbf{0})}\frac{1}{\Vert \mathbf{x}-\mathbf{y}\Vert ^{2\alpha }}d%
\mathbf{y}d\mathbf{x}  \notag \\
&\leq &\int_{\mathcal{B}_{R_{1}}(\mathbf{0})}\int_{\mathcal{B}_{R_{1}}(%
\mathbf{0})}\frac{1}{\Vert \mathbf{x}-\mathbf{y}\Vert ^{2\alpha }}d\mathbf{y}%
d\mathbf{x}  \notag
\end{eqnarray}
\begin{eqnarray}
&=&\mbox{Tr}\left( [\mathcal{K}_{\alpha }^{\mathcal{B}_{R_{1}}(\mathbf{0}%
)}]^{2}\right) =R_{1}^{4-2\alpha }\mbox{Tr}\left( [\mathcal{K}_{\alpha }^{%
\mathcal{B}(\mathbf{0})}]^{2}\right) =\frac{R_{1}^{4-2\alpha }2^{2-2\alpha
+1}\pi ^{2-\frac{1}{2}}\Gamma (\frac{2-2\alpha +1}{2})}{(2-2\alpha )\Gamma
(2-\alpha +1)}<\infty .  \notag
\end{eqnarray}

\noindent From Theorem \ref{par}(i), equation (\ref{6.7.1}) is finite for $%
D= \mathcal{B}_{R_{1}}(\mathbf{0})\backslash \mathcal{B}_{R_{2}}(\mathbf{0}).
$ Similarly, these computations can be extended to the finite union of
circular rings.

\begin{remark}
\textrm{Note that for a ball} $D=\mathcal{B}_{1}(\mathbf{0})=\mathcal{B}(%
\mathbf{0})=\{ \mathbf{x}\in \mathbb{R}^{d};\ \|\mathbf{x}\| \leq
1\},$ \textrm{the function} $\vartheta (\boldsymbol{\lambda })$
\textrm{in} (\ref{UD1}) \textrm{is of the form}
$\int_{\mathcal{B}_{1}(\mathbf{0})} \exp\left(\mathrm{i}\left\langle \mathbf{x},%
\boldsymbol{\lambda }\right\rangle\right)d\mathbf{x}=(2\pi )^{d/2} \frac{%
\mathcal{J}_{d/2}\left( \|\boldsymbol{\lambda }\|\right)}{\|\boldsymbol{%
\lambda }\|^{d/2}},$ \textrm{for} $d\geq 2,$ \textrm{where}
$\mathcal{J}_{\nu }(\mathbf{z})$ \textrm{is the Bessel function of
the first kind and order} $\nu > -1/2.$ \textrm{For a rectangle,}
$D=\prod =\left\{ a_{i}\leq x_{i}\leq b_{i},\ i=1,\dots,d\right\},$
\textrm{with} $\mathbf{0}\in \prod,$ \textrm{we have} $\vartheta
(\boldsymbol{\lambda })= \prod_{j=1}^{d}\left( \exp\left(
\mathrm{i}\lambda_{j}b_{j}\right)-\exp\left(
\mathrm{i}\lambda_{j}a_{j}\right)\right)/\mathrm{i}\lambda_{j},$ for
$d\geq 1.$ \textrm{Moreover for $d=2,$ considering the non-convex set} $D=\mathcal{B}_{1}(\mathbf{0}%
)\cup \mathcal{B}_{1}((2,0))\subset \mathbb{R}^{2},$
\begin{equation*}
\vartheta (\boldsymbol{\lambda })=\vartheta (\lambda_{1},\lambda_{2})=\int_{%
\mathcal{B}_{1}(\mathbf{0})\cup \mathcal{B}_{1}((2,0))}
\exp\left(\mathrm{i}\left\langle \mathbf{x},\boldsymbol{\lambda
}\right\rangle\right)d\mathbf{x}= \frac{2\pi
\mathcal{J}_{1}(\|\boldsymbol{\lambda }\|)}{\|\boldsymbol{\lambda }\|}%
\left(1+\exp \left( 2\mathrm{i}\lambda_{1}\right)\right),
\end{equation*}
\noindent \textrm{and for} $D=\mathcal{B}_{R_{1}}(\mathbf{0})\backslash \mathcal{B}%
_{R_{2}}(\mathbf{0}),$  $\vartheta (\boldsymbol{\lambda })= (2\pi R_{1})\mathcal{J}_{1}\left( \|%
\boldsymbol{\lambda }\|R_{1}\right)/\|\boldsymbol{\lambda }\|-(2\pi R_{2})%
\mathcal{J}_{1}\left( \|\boldsymbol{\lambda }\|R_{2}\right)/\|\boldsymbol{%
\lambda }\|.$

\end{remark}

\vspace*{1cm}

The following corollary is an extension of Proposition 2 in
\cite{Dobrushin}.

\begin{corollary}
\label{corr} \textit{Assume that the conditions of Theorem \ref{par} hold. Then, the
limit random variable $S_{\infty }$ admits the following series
representation:}

\begin{equation}
S_{\infty }\underset{d}{=}  c(d,\alpha )|D |\sum_{\mathbf{n}\in
\mathbb{N}^{d}_{*}}\mu_{\mathbf{n}}(\mathcal{H})(\varepsilon_{\mathbf{n}}^{2}-1)= \sum_{\mathbf{n}\in \mathbb{N}^{d}_{*}}\lambda_{\mathbf{n}}(S_{\infty})(%
\varepsilon_{\mathbf{n}}^{2}-1),  \label{eqlsum}
\end{equation}
\noindent \textit{where}  $$c(d,\alpha )=\frac{\Gamma
\left(\frac{d-\alpha}{2}\right)}{\pi^{d/2}2^{\alpha }\Gamma
(\alpha/2)}=\frac{1}{\gamma (\alpha )}$$ \textit{was already introduced in
 (\ref{eRc})}, $\varepsilon_{\mathbf{n}}$ \textit{are independent and
identically
distributed standard Gaussian random variables, and} $\mu_{\mathbf{n}}(%
\mathcal{H}),$ $\mathbf{n}\in \mathbb{N}^{d}_{*},$ \textit{is a sequence of
non-negative real numbers, which are the eigenvalues of the self-adjoint
Hilbert-Schmidt operator}
\begin{equation}
\mathcal{H}(h)(\boldsymbol{\lambda
}_{1})=\int_{\mathbb{R}^{d}}H_{1}\left( \boldsymbol{\lambda
}_{1}-\boldsymbol{\lambda }_{2}\right) h\left(
\boldsymbol{\lambda }_{2}\right) G_{\alpha }(d\boldsymbol{\lambda }%
_{2}):L^{2}_{E}(\mathbb{R}^{d},G_{\alpha }) \longrightarrow
L^{2}_{E}(\mathbb{R}^{d},G_{\alpha }), \label{hso}
\end{equation}%
\noindent \textit{with}
\begin{equation}
G_{\alpha }(d\mathbf{x})=\frac{1}{\|\mathbf{x}\|^{d-\alpha
}}d\mathbf{x}, \label{kernelGG}
\end{equation}

\noindent \textit{and} $L^{2}_{E}(\mathbb{R}^{d},G_{\alpha })$
\textit{denoting the collection of linear combinations, with
real-valued coefficients, of complex-valued and Hermitian functions,
that are square integrable with respect to $G_{\alpha
}(d\mathbf{x}).$ Note that $L^{2}_{E}(\mathbb{R}^{d}, G_{\alpha })$
is a real Hilbert space, endowed with the scalar product
$$\left\langle \psi_{1},\psi_{2}\right\rangle_{G_{\alpha
}}=\int_{\mathbb{R}^{d}}\psi_{1}(\mathbf{x})\overline{\psi_{2}(\mathbf{x})}G_{\alpha
}(d\mathbf{x})$$ \noindent (see \cite{Peccati}, pp.
159-161, for the case of $L^{2}_{E}(\mathbb{R}, d\beta)$ spaces}).
\textit{The symmetric kernel} $H_{1}\left(\boldsymbol{\lambda
}_{1}-\boldsymbol{\lambda }_{2}\right)=H(\boldsymbol{\lambda }_{1},
\boldsymbol{\lambda }_{2}),$ \textit{is defined from  $H$ introduced
in  equation (\ref{eqH}), in terms of the characteristic function
$K$ given in equation} (\ref{UD1}).
\end{corollary}

The proof can be seen in  Appendix A. Indeed, it constitutes an extension of  Proposition 2 in \cite{Dobrushin}.

In the following proposition the  explicit relationship between the
eigenvalues of $\mathcal{K}_{\alpha }$ and $\mathcal{H}$ is derived.
Note that
Theorem \ref{par}(i) provides the equality between the traces of operators $%
\frac{\mathcal{K}_{\alpha }^{2}}{[|D|c(d,\alpha )]^{2}}$ and $%
\mathcal{H}^{2},$ with, as before, $\mathcal{H}$ having kernel
$H(\cdot, \cdot) $ given  in equation (\ref{eqH}) .

\begin{proposition}
\label{twooperators}

\textit{The operators} $\mathcal{A}_{\alpha }:
L^{2}_{E}(\mathbb{R}^{d},G_{\alpha }) \longrightarrow
L^{2}_{E}(\mathbb{R}^{d},G_{\alpha })$
\begin{equation*}
\mathcal{A}_{\alpha }(f)(\boldsymbol{\lambda }_{1})= c(d,\alpha
)\int_{\mathbb{R}^{d}}H_{1}\left( \boldsymbol{\lambda }_{1}
-\boldsymbol{\lambda }_{2}\right) f\left( \boldsymbol{\lambda
}_{2}\right) G_{\alpha }(d\boldsymbol{\lambda }_{2}),
\end{equation*}
\noindent \textit{and} $|D|^{-1}\mathcal{K}_{\alpha}: L^{2}(D
)\longrightarrow L^{2}(D )$ \textit{have the same eigenvalues. Here,}
$c(d,\alpha )$ \textit{was already introduced in
 (\ref{eRc})}, $H_{1}\left(\boldsymbol{\lambda }_{1}-\boldsymbol{\lambda }_{2}\right)=H(\boldsymbol{\lambda
}_{1},\boldsymbol{\lambda }_{2})$ \textit{with kernel} $H$ \textit{being given in
equation (\ref{eqH})},
 $G_{\alpha }$ \textit{is introduced in} (\ref{kernelGG}), \textit{and} $\mathcal{K}_{\alpha}$ \textit{is defined in
(\ref{RKD})}.
\end{proposition}

The proof of this result is also given in Appendix A. (See \cite{Veillette}, for the case $d=1$).

\begin{corollary}
\label{lfes} \textit{Let} $\{ \lambda_{k}(S_{\infty}),\ k\geq 1\}$
\textit{be the eigenvalues appearing in  representation
(\ref{eqlsum}),} \textit{arranged into a decreasing order of the magnitudes of their
modulus.} \textit{Then,  Theorem \ref{WaeK} holds for
this system of eigenvalues.}
\end{corollary}

The proof directly follows from Corollary \ref{corr}, Proposition \ref{twooperators} and Theorem \ref{WaeK}.

\section{Properties of Rosenblatt-type distribution}

\label{sec7} This section provides the L\'evy-Khintchine representation of the limit random variable $%
S_{\infty}$ (see \cite{Veillette}, for $d=1,$ in the discrete time
case), as well as its membership to a subclass of selfdecomposable
distributions, given by the Thorin class. The absolute continuity of the law
of $S_{\infty },$ and the boundedness of its probability density is then
obtained.

It is well-known that the distribution of a random variable $X$ is
infinitely divisible if for any integer $n\geq 1,$ there exist $X_{j}^{(n)},$
$j=1,2,\dots ,n,$ independent and identically distributed (i.i.d.) random
variables such that $X\underset{d}{=}X_{1}^{(n)}+\dots +X_{n}^{(n)}.$ Let $%
\mathcal{ID}(\mathbb{R})$ be the class of infinitely divisible
distributions or random variables. Recall that the cumulant function
of an infinitely divisible random variable $X$ admits the
L{\'e}vy-Khintchine representation

\begin{equation}
\log \mathrm{E}\left[ \exp \left( \mathrm{i}\theta X\right) \right] =\mathrm{%
i}a\theta -\frac{b}{2}\theta ^{2}+\int_{-\infty }^{\infty }(\exp (\mathrm{i}%
\theta u)-1-\mathrm{i}\tau (u)\theta )\mu (du),\quad \theta \in \mathbb{R},
\label{eqlkr}
\end{equation}%
\noindent where $a\in \mathbb{R},$ $b\geq 0,$ and
\begin{equation}
\tau (u)=\left\{
\begin{array}{l}
u\quad |u|\leq 1 \\
\frac{u}{|u|}\quad |u|>1,%
\end{array}%
\right.  \label{eqtau}
\end{equation}%
\noindent and where the L{\'e}vy measure $\mu $ is a Radon measure on $%
\mathbb{R}\setminus \{0\}$ such that $\mu (\{0\})=0$ and
\begin{equation*}
\int \min (u^{2},1)\mu (du)<\infty .
\end{equation*}%
An infinitely divisible random variable $X$ (or its law) is selfdecomposable if its characteristic function $\phi
(\theta )=E[\mathrm{i}\theta X],$ $\theta \in \mathbb{R},$ has the property that for every $c\in (0,1)$ there exists a
characteristic function $\phi _{c}$ such that $\phi (\theta )=\phi (c\theta )\phi _{c}(\theta ),$ $\theta \in
\mathbb{R}.$ It is known (see \cite{Sato}, p.95, Corollary 15.11) that an infinitely divisible law is selfdecomposable
if its L{\'e}vy measure has a density $q$ satisfying
\begin{equation*}
q(u)=\frac{h(u)}{|u|},\quad u\in \mathbb{R},
\end{equation*}%
\noindent with $h(u)$ being increasing on $(-\infty ,0)$ and decreasing on $%
(0,\infty ).$ Let $\mathcal{SD}(\mathbb{R})$ be the class of
selfdecomposable distributions or random variables. If $Y\in
\mathcal{SD}(\mathbb{R})$ then (see \cite{Jurek})
\begin{equation}
Y\underset{d}{=}\int_{0}^{\infty }\exp (-s)dZ(s)\underset{d}{=}%
\int_{0}^{\infty }\exp (-s\lambda )dZ(s\lambda ),\quad \lambda >0,
\label{LP}
\end{equation}%
\noindent where $\{Z(t),\ t\geq 0\}$ is a L{\'e}vy process whose law is
determined by that of $Y.$

We next define the Thorin class on $\mathbb{R}$ (see \cite{Thorin};
\cite{Barndorff}; \cite{James}) as follows:
We refer to $\gamma x$ as an \emph{elementary gamma random variable} if $x$
is nonrandom non-zero vector in $\mathbb{R},$ and $\gamma $ is a gamma
random variable on $\mathbb{R}_{+}.$ Then, the Thorin class on $\mathbb{R}$
(or the class of extended generalized gamma convolutions), denoted by $T(%
\mathbb{R}),$ is defined as the smallest class of distributions that
contains all elementary gamma distributions on $\mathbb{R},$ and is closed
under convolution and weak convergence. It is known that $T(\mathbb{R}%
)\subset \mathcal{SD}(\mathbb{R})\subset \mathcal{ID}(\mathbb{R}),$ and
inclusions are strict. Since any selfdecomposable distribution on $\mathbb{R}
$ is absolutely continuous (see, for instance, Example 27.8 in \cite{Sato})
and is unimodal (by \cite{Yamazato}; see also Theorem 53.1 in \cite{Sato}),
then, any selfdecomposable distribution has a bounded density function.

If a probability distribution function $F$ belongs to $T(\mathbb{R}),$ then,
its characteristic function has the form (see \cite{Thorin},
\cite{Barndorff})

\begin{equation}
\phi (\theta )=\exp \left( \mathrm{i}\theta a-\frac{b\theta ^{2}}{2}-\int_{%
\mathbb{R}}\left[ \log \left( 1-\frac{\mathrm{i}\theta }{u}\right) +\frac{%
\mathrm{i}u\theta }{1+u^{2}}\right] U(du)\right) ,  \label{mgf}
\end{equation}%
\noindent where $a\in \mathbb{R},$ $b\geq 0,$ and $U(du)$ is a
non-decreasing measure on $\mathbb{R}\backslash \{0\},$ called Thorin
measure, such that
\begin{equation*}
U(0)=0,\quad \int_{-1}^{1}\left\vert \log |u|\right\vert U(du)<\infty ,\quad
\int_{-\infty }^{-1}\frac{1}{u^{2}}U(du)+\int_{1}^{\infty }\frac{1}{u^{2}}%
U(du)<\infty .
\end{equation*}

The L\'evy density of a distribution from Thorin class is such that
\begin{equation}
|u|q(u)=\left\{
\begin{array}{l}
\int_{0}^{\infty }\exp(-yu)U(dy),\quad u>0 \\
\\
\int_{0}^{\infty }\exp(yu)U(dy),\quad u<0, \\
\end{array}%
\right.  \label{thorinclass}
\end{equation}
\noindent where $U(du)$ is the Thorin measure. In other words, the L\'evy
density is of the form $h(|u|)/|u|,$ where $h(|u|)=h_{0}(r),$ $r\geq 0,$ is
a completely monotone function over $(0,\infty ).$

The following result establishes the L\'evy-Khintchine representation of $%
S_{\infty },$ as well as the asymptotic orders at zero and at infinity of
its associated L\'evy density. The membership to the Thorin
self-decomposable subclass is then obtained. As a direct consequence, we
then have the existence and boundedness of the probability density of $%
S_{\infty }$ (see, for instance, Example 27.8 in \cite{Sato}).

\begin{theorem}
\label{th1} Let $S_{\infty }$ be given as in Theorem \ref{pr1} with $%
0<\alpha <d/2.$ Let us consider $\lambda_{k}(S_{\infty }),$ $k\geq 1,$ the
sequence of eigenvalues introduced in Corollary \ref{corr} satisfying the
properties stated in Theorem \ref{WaeK} (see Corollary \ref{lfes}). Then,

\begin{itemize}
\item[(i)] $S_{\infty }\in \mathcal{ID}(\mathbb{R})$ with the following L{\'e%
}vy-Khintchine representation:
\begin{equation}
\phi (\theta )=E[\mathrm{i}\theta S_{\infty }]=\exp \left(
\int_{0}^{\infty }\left( \exp (\mathrm{i}u\theta
)-1-\mathrm{i}u\theta \right) \mu _{\alpha /d}(du)\right),
\label{eq45}
\end{equation}%
\noindent where $\mu _{\alpha /d}$ is supported on $(0,\infty )$ having
density
\begin{equation}
q_{\alpha /d}(u)=\frac{1}{2u}\sum_{k=1}^{\infty }\exp \left( -\frac{u}{%
2\lambda _{k}(S_{\infty })}\right) ,\quad u>0.  \label{levydensity}
\end{equation}%
Furthermore, $q_{\alpha /d}$ has the following asymptotics as $%
u\longrightarrow 0^{+}$ and $u\longrightarrow \infty ,$
\begin{eqnarray}
&&q_{\alpha /d}(u)\sim \frac{[\widetilde{c}(d,\alpha )|D
|^{(d-\alpha )/d}]^{1/(1-\alpha /d)}\Gamma \left( \frac{1}{1-\alpha
/d}\right)
\left( \frac{u}{2}\right) ^{-1/(1-\alpha /d)}}{2u[(1-\alpha /d)]}  \notag \\
&=&\frac{2^{\frac{\alpha /d}{1-\alpha /d}}[\widetilde{c}(d,\alpha )|D |^{(d-\alpha )/d}]^{1/(1-\alpha /d)}\Gamma \left( \frac{1}{1-\alpha /d}%
\right) u^{\frac{(\alpha /d)-2}{(1-\alpha /d)}}}{[(1-\alpha
/d)]}\quad \mbox{as}\
u\longrightarrow 0^{+},  \notag \\
&&q_{\alpha /d}(u)\sim \frac{1}{2u}\exp (-u/2\lambda _{1}(S_{\infty
})),\quad \mbox{as}\ u\longrightarrow \infty ,  \label{eqDCT2th}
\end{eqnarray}%
\noindent where $\widetilde{c}(d,\alpha )$ is defined as in equation (\ref%
{eqctilde}).

\item[(ii)] $S_{\infty }\in \mathcal{SD}(\mathbb{R}),$ and hence it has a
bounded density.

\item[(iii)] $S_{\infty }\in T(\mathbb{R}),$ with Thorin measure given by
\begin{equation*}
U(dx)=\frac{1}{2}\sum_{k=1}^{\infty }\delta_{\frac{1}{2\lambda_{k}(S_{\infty
})}}(x),
\end{equation*}
\noindent where $\delta_{a}(x)$ is the Dirac delta-function at point $a.$

\item[(iv)] $S_{\infty }$ admits the integral representation
\begin{equation}
S_{\infty}\underset{d}{=}\int_{0}^{\infty } \exp\left(-u\right)d\left(
\sum_{k=1}^{\infty }\lambda_{k}(S_{\infty })A^{(k)}(u)\right)\underset{d}{=}%
\int_{0}^{\infty }\exp\left(-u\right)dZ(u),  \label{Thclass}
\end{equation}
\noindent where
\begin{equation}
Z(t)=\sum_{k=1}^{\infty }\lambda_{k}(S_{\infty })A^{(k)}(t),\quad t\geq 0,
\label{Thclass2}
\end{equation}
\noindent with $A^{(k)},$ $k\geq 1,$ being independent copies of a L\'evy
process.
\end{itemize}
\end{theorem}

\begin{proof}
\noindent (i) The proof  follows from Theorem \ref{WaeK},
equation (\ref{eqfa}), Corollary \ref{lfes}, and Lemma \ref{lem1}
below (see Appendix B), in a similar way to Theorem 4.2 in \cite{Veillette}. Specifically, let us first consider a truncated
version of the random series representation (\ref{eqlsum})
$$S_{\infty }^{(M)}=\sum_{k=1}^{M}\lambda_{k}(S_{\infty })(\varepsilon_{k}^{2}-1),$$
\noindent with $S_{\infty }^{M}\underset{d}{\longrightarrow
}S_{\infty},$  as $M$ tends to infinity. From the L\'evy-Khintchine
representation of the chi-square distribution (see, for instance,
\cite{Applebaum}, Example 1.3.22),
\begin{eqnarray}
E\left[\exp(\mathrm{i}\theta S_{\infty}^{(M)})\right]&=&\prod_{k=1}^{M}E\left[\exp\left(
\mathrm{i}\theta\lambda_{k}(S_{\infty })(\varepsilon_{k}^{2}-1)
\right)\right]\nonumber\\
&=&\prod_{k=1}^{M}\exp\left( -\mathrm{i}\theta \lambda_{k}(S_{\infty })+\int_{0}^{\infty }(\exp(\mathrm{i}\theta
u)-1)\left[
\frac{\exp\left(-u/(2\lambda_{k}(S_{\infty }))\right)}{2u}\right]du\right)\nonumber\\
&=&\prod_{k=1}^{M}\exp\left(  \int_{0}^{\infty }(\exp(\mathrm{i}\theta u)-1-\mathrm{i}\theta
u)\left[\frac{\exp(-u/2\lambda_{k}(S_{\infty }))}{2u}\right]du\right)\nonumber\\
&=&\exp\left(  \int_{0}^{\infty }(\exp(\mathrm{i}\theta u)-1-\mathrm{i}\theta u)\left[\frac{1}{2u}G_{\lambda (\alpha /d
)}^{(M)}\left(\exp(-u/2)\right)\right]du\right), \label{integrand}
\end{eqnarray}
\noindent where $G_{\lambda (\alpha
/d)}^{(M)}(x)=\sum_{k=1}^{M}x^{[\lambda_{k}(S_{\infty })]^{-1}}.$ To
apply the Dominated Convergence Theorem, the following upper bound
is used:
\begin{eqnarray}
\left|(\exp(\mathrm{i}\theta u)-1-\mathrm{i}\theta u)\left[ \frac{1}{2u}G_{\lambda (\alpha /d
)}^{(M)}\left(\exp(-u/2)\right)\right]\right| &\leq & \frac{\theta^{2}}{4}uG_{\lambda (\alpha
/d)}^{(M)}\left(\exp(-u/2)\right) \nonumber\\ &\leq & \frac{\theta^{2}}{4}uG_{\lambda (\alpha /d
)}\left(\exp(-u/2)\right),\nonumber\\ \label{equb}
\end{eqnarray}
\noindent where, as indicated in \cite{Veillette}, we have applied the inequality
$|\exp(\mathrm{i}z)-1-z|\leq \frac{z^{2}}{2},$ for $z\in \mathbb{R}.$ The right-hand side of (\ref{equb}) is
continuous, for $0<u<\infty,$ and from
 Theorem \ref{WaeK}, equation (\ref{eqfa}), Corollary \ref{lfes}, and Lemma \ref{lem1} in Appendix B, we obtain
\begin{eqnarray}
uG_{\lambda (\alpha /d)}\left(\exp(-u/2)\right) &\sim &
u\exp(-u/2\lambda_{1}(S_{\infty })),\quad \mbox{as}\
u\longrightarrow
\infty\nonumber\\
 uG_{\lambda (\alpha /d
)}\left(\exp(-u/2)\right) &\sim & [\widetilde{c}(d,\alpha )|D
|^{1-\alpha /d}]^{1/1-\alpha /d}\frac{u}{(1-\alpha /d)}
\nonumber\\
& &\Gamma \left( \frac{1}{1-\alpha /d
}\right)(1-\exp(-u/2))^{-1/(1-\alpha /d)} \nonumber\\
& & \sim Cu^{-\frac{\alpha /d}{1-\alpha /d}}\quad \mbox{as}\
u\longrightarrow 0, \label{asympt}
\end{eqnarray}
\noindent for some constant $C.$ Since $0<\frac{\alpha /d}{1-\alpha
/d}<1,$ equation (\ref{asympt}) implies that the right-hand side of
(\ref{equb}), which  does not depend on $M,$  is integrable on
$(0,\infty ).$ Hence, by the Dominated Convergence Theorem, as
$M\rightarrow \infty,$
\begin{eqnarray}
& & E\left[\exp(\mathrm{i}\theta S_{\infty}^{(M)})\right]\longrightarrow E\left[\exp(\mathrm{i}\theta S_{\infty})\right] \nonumber\\
& &=\exp\left( \int_{0}^{\infty }(\exp(\mathrm{i}\theta u)-1-\mathrm{i}\theta u)\left[\frac{1}{2u}G_{\lambda (\alpha /d
)}\left(\exp(-u/2)\right)\right]du\right), \label{eqDCT}
\end{eqnarray}
 \noindent which proves that equations
(\ref{eq45}) and (\ref{levydensity}) hold.

Again, from  Theorem \ref{WaeK}, equation (\ref{eqfa}), Corollary
\ref{lfes}, and Lemma \ref{lem1} below,

\begin{eqnarray}
& & \frac{1}{2u}G_{\lambda (\alpha /d)}\left(\exp(-u/2)\right)\sim
[\widetilde{c}(d,\alpha )|D |^{1-\alpha /d}]^{1/(1-\alpha
/d)}\frac{\Gamma \left( \frac{1}{1-\alpha /d
}\right)\left(\frac{u}{2}\right)^{-1/(1-\alpha /d)}}{2u[(1-\alpha
/d)]}\nonumber\\
& &=  \frac{2^{\frac{\alpha /d}{1-\alpha /d}}[\widetilde{c}(d,\alpha
)|D |^{1-\alpha /d}]^{1/(1-\alpha /d)}\Gamma \left(
\frac{1}{1-\alpha /d }\right)u^{\frac{(\alpha /d)-2}{(1-\alpha
/d)}}}{[(1-\alpha /d)]}\quad \quad \mbox{as}\
u\longrightarrow 0\nonumber\\
& & \frac{1}{2u}G_{\lambda (\alpha /d)}\left(\exp(-u/2)\right)\sim
\frac{1}{2u}\exp(-u/2\lambda_{1}(S_{\infty }))\quad \mbox{as}\
u\longrightarrow \infty. \label{eqDCT2}
\end{eqnarray}

Thus, equation (\ref{eqDCT2}) provides the  asymptotic orders given
in (\ref{eqDCT2th}).

\medskip

\noindent (ii) From (i), it follows that  $S_{\infty }\in
\mathcal{SD}(\mathbb{R}),$ and hence it has a bounded density (see
\cite{Bondesson},  Example 27.8 in \cite{Sato}, and \cite{Yamazato}).
Note that an alternative proof of the boundedness of the probability
density of $S_{\infty }$ is provided in Appendix C, where an upper
bound is also obtained.

\medskip

\noindent (iii) In view of (\ref{thorinclass}) and
(\ref{levydensity}), $S_{\infty }\in T(\mathbb{R})$ with Thorin
measure given by \begin{equation}U(dx)=\frac{1}{2}\sum_{k=1}^{\infty
}\delta_{\frac{1}{2\lambda_{k}(S_{\infty })}}(x),\label{im}
\end{equation}
\noindent where $\delta_{a}(x)$ is the Dirac delta-function at point
$a.$ From Theorem \ref{WaeK}, Corollaries \ref{corr} and
Proposition \ref{twooperators} (see also Corollary \ref{lfes}), the
number of terms in the sum (\ref{im}) is infinite. Hence,  the
Thorin measure $U(dx),$ as a counting measure, has infinite total
mass.
 The form of Thorin measure is a direct
consequence of (\ref{thorinclass}) and (\ref{levydensity}).

\medskip

\noindent (iv) As in \cite{Maejima}, we consider a gamma
subordinator $\gamma_{\lambda }(t),$ $t\geq 0,$ with parameter
$\lambda >0,$ that is, a L\'evy process such that $\gamma_{\lambda
}(0)=0,$ and $P\left\{ \gamma_{\lambda }(t)\in dx\right\}=\lambda
^{-t}\Gamma ^{-1}(t)\exp(-x\lambda )x^{t-1}dx,$ $x>0,$ and a
homogeneous Poisson process $N(t),$ $t\geq 0,$ with unit rate.
Assume that the two processes are independent. Then (see \cite{Aoyama}), for any $c>0,$ and $\lambda >0,$ the
representation (\ref{LP}) (see \cite{Jurek}) can be specified as follows:
$$\gamma_{\lambda }(c)\underset{d}{=}\int_{0}^{\infty }\exp\left(-t\right)d\gamma_{\lambda }(N(ct)).$$
The process $A(t)=\gamma_{1/2}(N(t/2))-t,$ $t\geq 0,$ is a L\'evy
process.

For $k\geq 1,$ let us consider
$\gamma^{(k)}_{\frac{1}{2}}\left(\frac{1}{2}\right)$ and
$A^{(k)}(t)$ to be independent copies of
$\gamma_{\frac{1}{2}}\left(\frac{1}{2}\right)$ and $A(t),$
respectively. Then, we have
$$\varepsilon_{k}^{2}-1\underset{d}{=}\gamma^{(k)}_{\frac{1}{2}}\left(\frac{1}{2}\right)\underset{d}{=}\int_{0}^{\infty }\exp\left(-u\right)
dA^{(k)}(u),$$ \noindent where $\varepsilon_{k}$ are independent and
identically distributed standard normal random variables as given in
 the series expansion (\ref{eqlsum}). Then, for
$\lambda_{k}(S_{\infty }),$ $k\geq 1,$ being the eigenvalues
appearing in such a series expansion,  arranged into a decreasing
order of their magnitudes, we obtain that the distribution of
$S_{\infty }$ admits the integral representation (\ref{Thclass}),
with, $A^{k},$ $k\geq 1,$ in equation (\ref{Thclass2}) being
independent copies of the L\'evy process
$A(t)=\gamma_{1/2}(N(t/2))-t,$ $t\geq 0.$

\end{proof}

For any $0<\alpha /d<1/2,$ the L{\'e}vy measure $\mu _{\alpha /d}$ satisfies
\begin{equation*}
\int_{0}^{\infty }u^{2}\mu _{\alpha /d}(du)=\mathrm{E}[S_{\infty
}^{2}]=[a_{d}(D)]^{2}.
\end{equation*}%
Furthermore, when $\alpha /d\longrightarrow 1/2,$ since $(\exp (\mathrm{i}\theta u)-1-\mathrm{i}\theta u)\rightarrow
(-1/2)\theta ^{2}$ (see \cite{Veillette}), we have
\begin{equation*}
\phi (\theta )=\exp \left( \int_{0}^{\infty }\frac{\exp (\mathrm{i}\theta
u)-1-\mathrm{i}\theta u}{u^{2}}u^{2}\mu _{\alpha /d}(du)\right)
\longrightarrow \exp \left( -\frac{1}{2}\theta ^{2}\right) ,
\end{equation*}%
which means that $S_{\infty }\longrightarrow N(0,1).$

In addition, from Theorem \ref{th1}, it can be proved, in a similar way to
Corollary 4.3 and 4.4  in \cite{Veillette}, that, for $0<\alpha
/d<1/2,$ the probability density function of $S_{\infty }$ is infinitely
differentiable with all derivatives tending to $0$ as $|x|\longrightarrow
\infty.$ Also, the following inequality holds
\begin{equation*}
P[S_{\infty }<-x]\leq \exp\left(-\frac{1}{2}x^{2}\right),\quad x>0.
\end{equation*}
\noindent We also note that, for $\epsilon >0,$
\begin{equation*}
\lim_{u\rightarrow \infty }\frac{P[S_{\infty }>u+\epsilon]}{P[S_{\infty }>u]}%
=\exp\left( -\frac{\epsilon}{2\lambda_{1}(S_{\infty})}\right).
\end{equation*}

\begin{remark}
\textrm{{In view of the integral representation (\ref{Thclass}), one
can construct an Ornstein-Uhlenbeck type process
\begin{equation*}
dS(t)=-\lambda S(t)+d L(\lambda t),\quad t\geq 0,\quad \lambda >0,
\end{equation*}
\noindent driven by a L\'evy process $L(t),$ $t\geq 0,$ and with
marginal Rosenblatt distribution $S_{\infty }.$ The driving process
$L(t)$ is referred to as
the background  L\'evy process, and it is introduced in (\ref%
{Thclass2}).} }
\end{remark}

\section{Appendices}

\section*{Appendix A}

\subsection*{Proof of Corollary \protect\ref{corr}}

From condition (\ref{6.7.1}),  operator $\mathcal{H}$ is a
Hilbert-Schmidt operator from $L^{2}_{E}(\mathbb{R}^{d},G_{\alpha
})$ into $L^{2}_{E}(\mathbb{R}^{d},G_{\alpha }),$ which
admits a spectral decomposition, in terms of a sequence of eigenvalues $%
\{\mu_{\mathbf{n}}(\mathcal{H}), \ \mathbf{n}\in
\mathbb{N}^{d}_{*}\},$ and a complete orthonormal system of
eigenvectors $\{\varphi_{\mathbf{n}},\ \mathbf{n}\in
\mathbb{N}^{d}_{*}\}$ of $L^{2}_{E}(\mathbb{R}^{d},G_{\alpha }),$ as
follows:

\begin{equation}
H_{1}(\mathbf{x}-\mathbf{y})=H(\mathbf{x},\mathbf{y})=\sum_{\mathbf{n}\in \mathbb{N}^{d}_{*}}\mu_{%
\mathbf{n}}(\mathcal{H})\varphi_{\mathbf{n}}(\mathbf{x})\overline{\varphi_{%
\mathbf{n}}(\mathbf{y})},
\label{eqhsoexp}
\end{equation}
\noindent where convergence holds in the
$L^{2}_{E}(\mathbb{R}^{d},G_{\alpha })\otimes
L^{2}_{E}(\mathbb{R}^{d},G_{\alpha })$ sense, i.e.,
\begin{equation}
\left\|H(\mathbf{x},\mathbf{y})-\sum_{\mathbf{n}\in \mathbb{N}^{d}_{*}}\mu_{%
\mathbf{n}}(\mathcal{H})\varphi_{\mathbf{n}}\otimes \overline{\varphi}_{%
\mathbf{n}}\right\|_{L^{2}_{E}(\mathbb{R}^{d},G_{\alpha })\otimes
L^{2}_{E}(\mathbb{R}^{d},G_{\alpha })}^{2}=0. \label{nordiff}
\end{equation}

We can establish the following isometry $\widehat{\mathcal{I}}_{2}$
between the Hilbert space $L^{2}_{E}(\mathbb{R}^{d},G_{\alpha
})\otimes L^{2}_{E}(\mathbb{R}^{d},G_{\alpha }),$ and the two-Wiener
chaos   of the isonormal process $X$ on
$H=L^{2}_{E}(\mathbb{R}^{d},G_{\alpha }),$ given by
\begin{equation}X: h\in  L^{2}_{E}(\mathbb{R}^{d},G_{\alpha })\longrightarrow X(h)=\int_{\mathbb{R}^{d}}^{\prime}h(\mathbf{x})\frac{%
Z\left( d\boldsymbol{x} \right)}{\|\mathbf{x}\|^{(d-\alpha
)/2}}\label{isoprocc} \end{equation} \noindent (see \cite{Peccati}, Chapter 9), considering the following identification
between orthonormal bases of both spaces: For a given orthonormal
basis $\{\varphi_{\mathbf{n}}\otimes
\overline{\varphi_{\mathbf{n}}},\ \mathbf{n}\in
\mathbb{N}^{d}_{*}\}$ of $L^{2}_{E}(\mathbb{R}^{d},G_{\alpha
})\otimes L^{2}_{E}(\mathbb{R}^{d},G_{\alpha }),$ its image by such
an isometry $\widehat{\mathcal{I}}_{2}$ is defined as
\begin{equation}\widehat{\mathcal{I}}_{2}(\varphi_{\mathbf{n}}\otimes \overline{\varphi_{\mathbf{n}}})=\int_{\mathbb{R}^{2d}}^{\prime \prime}\left[\varphi_{\mathbf{n}}(\mathbf{x}%
_{1})\overline{\varphi_{\mathbf{n}}(\mathbf{x}_{2})}\right]\frac{Z\left( d\mathbf{x}%
_{1}\right)}{\|\mathbf{x}_{1}\|^{(d-\alpha)/2}} \frac{Z\left( d\mathbf{x}%
_{2}\right)}{\|\mathbf{x}_{2}\|^{(d-\alpha)/2}},  \label{isometry}
\end{equation}

\noindent which also defines an orthonormal basis in the  two-Wiener
chaos   of the isonormal process $X$ in (\ref{isoprocc}).

From equations (\ref{eqhsoexp})--(\ref{nordiff}), by the
orthonormality of the eigenvector basis $\varphi_{n}\otimes
\overline{\varphi}_{n}$ of $L^{2}_{E}(\mathbb{R}^{d},G_{\alpha
})\otimes L^{2}_{E}(\mathbb{R}^{d},G_{\alpha }),$
\begin{eqnarray}
\left\langle H, \varphi_{\mathbf{k}}\otimes
\overline{\varphi}_{\mathbf{k}}\right\rangle_{L^{2}_{E}(\mathbb{R}^{d},G_{\alpha
})\otimes
L^{2}_{E}(\mathbb{R}^{d},G_{\alpha })}&=&\left\langle\sum_{\mathbf{n}\in \mathbb{N}^{d}_{*}}\mu_{%
\mathbf{n}}(\mathcal{H})\varphi_{\mathbf{n}}\otimes
\overline{\varphi}_{\mathbf{n}},\varphi_{\mathbf{k}}\otimes
\overline{\varphi_{\mathbf{k}}}\right\rangle_{L^{2}_{E}(\mathbb{R}^{d},G_{\alpha
})\otimes
L^{2}_{E}(\mathbb{R}^{d},G_{\alpha })}\nonumber\\
&=& \mu_{%
\mathbf{k}}(\mathcal{H}),\quad \forall \mathbf{k}\in \mathbb{N}^{d}_{*}.
\label{innerproduct}
\end{eqnarray}
Again, from equations (\ref{eqhsoexp})--(\ref{nordiff}), and
(\ref{innerproduct}), considering the isometry
$\widehat{\mathcal{I}}_{2}$ in (\ref{isometry})
\begin{eqnarray}
& & \int_{\mathbb{R}^{2d}}^{\prime \prime}H(\mathbf{x}_{1},\mathbf{x}_{2})
\frac{Z\left( d\mathbf{x}_{1}\right)}{\|\mathbf{x}_{1}\|^{(d-\alpha)/2}}
\frac{Z\left( d\mathbf{x}_{2}\right)}{\|\mathbf{x}_{2}\|^{(d-\alpha)/2}}
\notag \\
&=&\sum_{\mathbf{n}\in \mathbb{N}^{d}_{*}}\mu_{\mathbf{n}}(\mathcal{H})\int_{%
\mathbb{R}^{2d}}^{\prime \prime}\left[\varphi_{\mathbf{n}}(\mathbf{x}%
_{1})\overline{\varphi_{\mathbf{n}}(\mathbf{x}_{2})}\right]\frac{Z\left( d\mathbf{x}%
_{1}\right)}{\|\mathbf{x}_{1}\|^{(d-\alpha)/2}} \frac{Z\left( d\mathbf{x}%
_{2}\right)}{\|\mathbf{x}_{2}\|^{(d-\alpha)/2}}  \notag \\
&=&\sum_{\mathbf{n}\in \mathbb{N}^{d}_{*}}\mu_{\mathbf{n}}(\mathcal{H}%
)H_{2}\left(\int_{\mathbb{R}^{d}}\varphi_{\mathbf{n}}(\mathbf{x})\frac{%
Z\left( d\mathbf{x} \right)}{\|\mathbf{x}\|^{(d-\alpha )/2}}\right),
\label{eqwchaos}
\end{eqnarray}
\noindent where $H_{2}$ denotes, as before, the second Hermite
polynomial. Note that  summation and integration can be swapped, in
view of the convergence of the series (\ref{eqhsoexp}) in the space
$L^{2}_{E}(\mathbb{R}^{d},G_{\alpha })\otimes
L^{2}_{E}(\mathbb{R}^{d},G_{\alpha }),$ and the referred isometry
between  $L^{2}_{E}(\mathbb{R}^{d},G_{\alpha })\otimes
L^{2}_{E}(\mathbb{R}^{d},G_{\alpha })$ and the two-Wiener chaos of
isonormal process $X$ introduced in  (\ref{isoprocc}) (see also
equations (\ref{nordiff})--(\ref{innerproduct})).

Note that
\begin{equation*}
\int_{\mathbb{R}^{2d}}^{\prime \prime}\varphi_{\mathbf{n}}(\mathbf{x})\frac{%
Z\left( d\boldsymbol{x} \right)}{\|\mathbf{x}\|^{(d-\alpha )/2}},\quad
\mathbf{n}\in \mathbb{N}^{d}_{*},
\end{equation*}
\noindent are independent zero-mean  Gaussian random variables with  variance $\int_{\mathbb{R}^{2d}}|\varphi_{%
\mathbf{n}}(\mathbf{x})|^{2}G_{\alpha }(d\mathbf{x}),$ due to the orthogonality of the functions $\varphi_{%
\mathbf{n}},$ $\mathbf{n}\in \mathbb{N}^{d}_{*},$ in the space
$L^{2}_{E}(\mathbb{R}^{d},G_{\alpha }).$ From equations
(\ref{6.7.3}) and (\ref{eqwchaos}),
\begin{equation*}
S_{\infty }\ \underset{d}{=} \ c(d,\alpha )|D |\sum_{\mathbf{n}\in \mathbb{N}%
_{*}}\mu_{\mathbf{n}}(\mathcal{H})(\varepsilon_{\mathbf{n}}^{2}-1).
\end{equation*}
\noindent Equation (\ref{eqlsum}) is then obtained by setting $\lambda_{%
\mathbf{n}}(S_{\infty})=c(d,\alpha )|D |\mu_{\mathbf{n}}(\mathcal{H}).$

\subsection*{Proof of Proposition \protect\ref{twooperators}}

Let us consider $\mathcal{F}$ and $\mathcal{F}^{-1}$ the
Fourier and inverse Fourier transforms respectively defined on $L^{1}(\mathbb{R}^{d})$ and $%
L^{2}(\mathbb{R}^{d}).$
  Consider an eigenpair $(\mu ,h)$ of the operator $%
\mathcal{A}_{\alpha},$ we have that $\int_{\mathbb{R}^{d}}|h(\mathbf{y})|^{2}%
\frac{1}{\|\mathbf{y}\|^{d-\alpha }}<\infty. $ Applying the inverse
Fourier  transform $\mathcal{F}$ to both sides of the identity
\begin{equation*}
\mu h=\mathcal{A}_{\alpha }h,
\end{equation*}
we get
\begin{equation*}
\mu \mathcal{F}^{-1}(h)= \mathcal{F}^{-1}(\mathcal{A}_{\alpha
}h)=c(d,\alpha )\mathcal{F}^{-1}(H_{1}*H_{2}),
\end{equation*}
\noindent where, as before,
\begin{equation*}
H_{1}(\boldsymbol{\lambda }_{1}-\boldsymbol{\lambda
}_{2})=H(\boldsymbol{\lambda }_{1},\boldsymbol{\lambda }_{2}),
\end{equation*}
\noindent with kernel $H$ being  defined in equation (\ref{eqH}), and $H_{2}(\mathbf{y})=\|\mathbf{y}%
\|^{-d+\alpha }h(\mathbf{y}).$  In the
computation of this inverse Fourier transform, we note that $H_{1}\in L^{1}(%
\mathbb{R}^{d})\cap L^{2}(\mathbb{R}^{d}).$ In order to apply the
convolution theorem, we first perform the following decomposition:
\begin{equation*}
H_{2}(\mathbf{y})=\|\mathbf{y}\|^{-d+\alpha }h(\mathbf{y})\mathbf{1}_{%
\mathcal{B}_{1}(\mathbf{0})}(\mathbf{y})+\|\mathbf{y}\|^{-d+\alpha }h(\mathbf{y})\mathbf{1%
}_{\mathbb{R}^{d}\backslash \mathcal{B}_{1}(\mathbf{0})}(\mathbf{y}):= H_{2}^{-}(\mathbf{y%
})+H_{2}^{+}(\mathbf{y}),
\end{equation*}
\noindent where $\mathcal{B}_{1}(\mathbf{0})$ denotes, as before,
the ball with center zero and radius one in $\mathbb{R}^{d}.$ Since
\begin{equation*}
\int_{\mathbb{R}^{d}}h^{2}(\mathbf{y})\|\mathbf{y}\|^{-d+\alpha }d\mathbf{y}%
<\infty,
\end{equation*}
\noindent $H_{2}^{-}\in L^{1}(\mathbb{R}^{d}),$ and $H_{2}^{+}\in L^{2}(%
\mathbb{R}^{d}).$ Applying the linearity of the convolution and
Fourier transform, the convolution theorem for both $L^{1}$  and
$L^{2}$ functions (see \cite{Triebel},  and \cite{Stade}) leads to
\begin{eqnarray}  \label{eqconv}
\mu \mathcal{F}^{-1}(h)&=& c(d,\alpha
)\mathcal{F}^{-1}(H_{1}*H_{2})= c(d,\alpha )
\left[ \mathcal{F}^{-1}(H_{1}*H_{2}^{-})+\mathcal{F}^{-1}(H_{1}*H_{2}^{+})%
\right]  \notag \\
&=&  c(d,\alpha )|D |^{-1}\mathbf{1}_{D
}(\mathcal{F}^{-1}(H_{2}^{-}+H_{2}^{+}))=  c(d,\alpha )|D
|^{-1}\mathbf{1}_{D }\mathcal{F}^{-1}H_{2},  \notag \\
\end{eqnarray}
\noindent where we have considered equations (\ref{UD1}) and
(\ref{eqH}). From (\ref{eqconv}), we can see that the support of
$\mathcal{F}^{-1}(h)$ is contained in $D,$ for any eigenfunction $h$
of $\mathcal{A}_{\alpha }.$ The convolution theorem for generalized
functions (see \cite{Triebel}) can be
applied again to $H_{2},$ since $h$ has compact support. By (\ref{kernelGG}%
), $G_{\alpha }(d\mathbf{x})=g_{\alpha }(\mathbf{x})d\mathbf{x},$ with $%
g_{\alpha }(\mathbf{x})=\|\mathbf{x}\|^{-d+\alpha }.$ Then,
\begin{equation*}
h(\mathbf{y})\|\mathbf{y}\|^{-d+\alpha }= \mathcal{F}\left(
\mathcal{F}^{-1}(h)*\mathcal{F}^{-1}(g_{\alpha
})\right)(\mathbf{y}).
\end{equation*}
\noindent Therefore, in equation (\ref{eqconv}), we obtain
\begin{eqnarray}
\mu \mathcal{F}^{-1}(h)&=& c(d,\alpha )|D |^{-1}\mathbf{1}_{D }\mathcal{F}^{-1}%
\left[ \mathcal{F}\left(
\mathcal{F}^{-1}(h)*\mathcal{F}^{-1}(g_{\alpha
})\right)\right]  \notag \\
&=& c(d,\alpha )|D|^{-1}\mathbf{1}_{D }\left(
\mathcal{F}^{-1}(h)*\mathcal{F}^{-1}(g_{\alpha })\right).
\label{eee}
\end{eqnarray}

The inverse Fourier transform $\mathcal{F}^{-1}$ of $g_{\alpha }(\mathbf{y})=\|\mathbf{y}%
\|^{-d+\alpha }$ is obtained from  equation (\ref{fist117}) in Lemma
\ref{l1s117} (see also \cite{Stein}, p.117):
\begin{equation*}
\mathcal{F}^{-1}(g_{\alpha })(\mathbf{z})= \frac{1}{c(d,\alpha)\|\mathbf{z}\|^{\alpha }}=\frac{\pi^{d/2}2^{\alpha
}\Gamma (\alpha/2)}{\Gamma \left(\frac{d-\alpha}{2}\right)}\|\mathbf{z}\|^{-\alpha }.
\end{equation*}
\noindent  Applying
(\ref{eee}) and this last relation, we finally obtain that, for an eigenpair $%
(\mu , h)$ of $\mathcal{A}_{\alpha },$
\begin{eqnarray}
\mu \mathcal{F}^{-1}(h)(\mathbf{z})&=& |D|^{-1}\mathbf{1}_{D}(\mathbf{z}%
)\int_{D} \|\mathbf{z}-\mathbf{y}\|^{-\alpha }\mathcal{F}^{-1}(h)(\mathbf{y}%
)d\mathbf{y},  \label{eqeigr}
\end{eqnarray}

\noindent since, as commented before, $\mathcal{F}^{-1}(h)$ is supported on $%
D .$ Thus, if $(\mu ,h)$ is an eigenpair of $\mathcal{A}_{\alpha },$ then $%
(\mu , \mathcal{F}^{-1}(h))$ is and eigenpair for
$|D|^{-1}\mathcal{K}_{\alpha }$ on $L^{2}(D).$ The converse
assertion also holds. Hence, there exists a one-to-one
correspondence between eigenpairs of $\mathcal{A}_{\alpha }$ and
$|D|^{-1}\mathcal{K}_{\alpha },$ which preserves the eigenvalues, as
we wanted to prove.

\section*{Appendix B}

The proof of Theorem \ref{th1} is based on  Lemma 4.1 in
\cite{Veillette}, which is now formulated.

\begin{lemma}
\label{lem1} \textit{Define the function
$G_{c}(x)=\sum_{k=1}^{\infty }x^{c_{k}},$
with $c=\{c_{n}\}$ being a positive strictly increasing sequence such that $%
c_{n}\sim \beta n^{\alpha },$ as $n\longrightarrow \infty,$ for some $%
1/2<\alpha <1,$ and constant $\beta >0.$ Then,}
\begin{eqnarray}
G_{c}(x) & \sim & x^{c_{1}},\quad \mbox{as} \ x\longrightarrow 0  \notag \\
G_{c}(x) & \sim & \frac{1}{\alpha \beta^{1/\alpha }}\Gamma \left( \frac{1}{%
\alpha }\right)(1-x)^{-1/\alpha }, \quad \mbox{as} \ x\longrightarrow 1.
\end{eqnarray}
\end{lemma}

\section*{Appendix C}

An alternative proof of the boundedness of the probability density of $%
S_{\infty },$ based on the series representation given in Corollary \ref%
{corr} is derived, and an upper bound is also provided.

\subsection*{Proof of boundedness of the probability density of $S_{\infty }$%
}
 From Corollary \ref{lfes}, there exist  two indexes $\mathbf{k}_{0}$ and
$\mathbf{k}_{1}$ such that $\lambda_{\mathbf{k}_{0}}(S_{\infty})>\lambda_{\mathbf{k}_{1}}(S_{\infty}).$
 Then,

\begin{equation*}
S_{\infty }=\sum_{\mathbf{k}\in \mathbb{N}^{d}_{*}}
\lambda_{\mathbf{k}}(S_{\infty})\left(
\varepsilon_{\mathbf{k}}^{2}-1\right)
=\lambda_{\mathbf{k}_{0}}(S_{\infty})(\varepsilon_{\mathbf{k}_{0}}^{2}-1)+\lambda_{\mathbf{k}_{1}}(S_{\infty})
(\varepsilon_{\mathbf{k}_{1}}^{2}-1)+\eta.
\end{equation*}
\noindent where
\begin{equation*}
\eta =\sum_{\mathbf{k}\in \mathbb{N}^{d}_{*},\mathbf{k}\neq
\mathbf{k}_{0}, \mathbf{k}_{1}}\lambda
_{\mathbf{k}}(S_{\infty})\left(
\varepsilon_{\mathbf{k}}^{2}-1\right) .
\end{equation*}%

Thus,
\begin{equation*}
S_{\infty }=\lambda _{\mathbf{k}_{1}}(S_{\infty})(\beta
\varepsilon_{\mathbf{k}_{0}}^{2}+\varepsilon_{\mathbf{k}_{1}}^{2})-(\lambda
_{\mathbf{k}_{0}}(S_{\infty})+\lambda
_{\mathbf{k}_{1}}(S_{\infty}))+\eta ,
\end{equation*}%
\noindent where $\beta =\lambda
_{\mathbf{k}_{0}}(S_{\infty})/\lambda_{\mathbf{k}_{1}}(S_{\infty}).$

The random variables $\varepsilon_{\mathbf{k}_{0}}^{2}$ and
$\varepsilon_{\mathbf{k}_{1}}^{2}$ are independent. Since the density of $\varepsilon _{\mathbf{k}_{1}\text{ }}^{2}$ is of the form%
\[
f_{\varepsilon _{\mathbf{k}_{1}\text{ }}^{2}}(x)=\frac{1}{\Gamma (\frac{1}{2})\sqrt{2%
}}x^{-1/2}e^{-x/2},\ x>0,
\]

\noindent and the density of $\beta \varepsilon _{\mathbf{k}_{0}\text{ }}^{2}$ is given by
\[
f_{\beta \varepsilon _{\mathbf{k}_{0}}^{2}}(x)=\frac{1}{\beta \Gamma (\frac{1}{2})%
\sqrt{2}}(x/\beta )^{-1/2}e^{-x/2\beta },\ x>0,
\]
\noindent  noting that $\beta =\frac{\lambda _{\mathbf{k}_{0}}(S_{\infty})}{\lambda _{\mathbf{k}_{1}}(S_{\infty})}>1,$
then
 the density of $\varsigma =\beta
\varepsilon_{\mathbf{k}_{0}}^{2}+\varepsilon_{\mathbf{k}_{1}}^{2}$ satisfies
\begin{eqnarray}
f_{\varsigma }(u)&=&\int_{0}^{u}f_{\varepsilon _{\mathbf{k}_{1}\text{ }%
}^{2}}(u-x)f_{\beta \varepsilon _{\mathbf{k}_{0}}^{2}}(x)dx \nonumber\\ &=&
\frac{e^{-u/2}}{2\Gamma ^{2}(\frac{1}{2})\sqrt{\beta }}%
\int_{0}^{u}(u-x)^{-1/2}e^{\frac{x}{2}}x^{-1/2}e^{-\frac{x}{2\beta }}dx =
\nonumber\\ & & \lbrack 1-\frac{1}{\beta }>0] \nonumber\\
&=&\frac{e^{-u/2}}{2\Gamma ^{2}(\frac{1}{2})\sqrt{\beta }}%
\int_{0}^{u}(u-x)^{-1/2}e^{\frac{x}{2}\left(1-\frac{1}{\beta }\right)}x^{-1/2}dx \nonumber\\ &\leq &
 \frac{e^{-u/2}e^{\frac{u}{2}\left(1-\frac{1}{\beta }\right)}}{2\Gamma ^{2}(\frac{1}{2})%
\sqrt{\beta }}\int_{0}^{u}(u-x)^{-1/2}x^{-1/2}dx \nonumber\\ &\leq &
e^{-\frac{u}{2\beta }}\frac{B(\frac{1}{2},\frac{1}{2})}{2\Gamma ^{2}(%
\frac{1}{2})\sqrt{\beta }}\leq \frac{1}{2\sqrt{\beta }}=\frac{1}{2\sqrt{%
\frac{\lambda _{\mathbf{k}_{0}}(S_{\infty})}{\lambda _{\mathbf{k}_{1}}(S_{\infty})}}}\leq \frac{1}{2}.
\end{eqnarray}
As the convolution of a bounded density with other is bounded, we then obtain the desired result.

\end{document}